\documentclass{amsart}
\usepackage{amsfonts,amssymb,amscd,amsmath,enumerate,extpfeil,mathtools,verbatim}
\usepackage[all]{xy}
\xyoption{curve}
\SelectTips{eu}{}

\usepackage[colorlinks]{hyperref}

\newcommand{\ub}{\underbracket[.5pt][2pt]}
\newcommand{\avlotimes}[1]{{}_{#1}{\!}\otimes}
\newcommand{\rotimes}[1]{\otimes{\!}_{#1}}
\newcommand{\im}{\operatorname{Im}}
\newcommand{\cls}{\operatorname{cls}}
\newcommand{\shift}{{\mathsf\Sigma}}
\newcommand{\susp}{{\varsigma}}
\newcommand{\nat}{{{}^{\natural}}}
\newcommand{\opp}{{}^{\mathsf{op}}}
\newcommand{\avop}{{}^{\mathsf{o}}}
\newcommand{\babar}[1]{{\mathsf{B}#1}}

\newcommand{\koszul}[1]{{{\operatorname K}{#1}}}
\newcommand{\priddy}[1]{{k\langle W^*\rangle}}
\newcommand{\cobar}[1]{{\mathsf{\Omega} #1}}
\newcommand{\DGM}[3]{{\mathsf{M}^{\mathsf{#1}}_{\mathsf{#2}}(#3)}}
\newcommand{\dcatdf}[3]{{\mathsf{D}^{\mathsf{#1}}_{\mathsf{#2}}(#3)}}

\newcommand{\ten}[2]{\mathsf T^{\mathsf#1}(#2)}
\newcommand{\tenn}[2]{{\mathsf T}^{#1}(#2)}
\newcommand{\ltensor}{\otimes^{\mathsf L}}
\newcommand{\rhom}{\operatorname{{\mathsf R}Hom}}

\newcommand{\HH}{\operatorname{H}}
\newcommand{\ZZ}{\operatorname{Z}}
\newcommand{\xra}{\xrightarrow}
\newcommand{\xla}{\xleftarrow}
\newcommand{\dd}{\partial}
\newcommand{\scup}{{\scriptstyle\smile}}
\newcommand{\ocup}{\dot{\scriptstyle\smile}}
\newcommand{\scap}{{\scriptstyle\frown}}
\newcommand{\ocap}{\dot{\scriptstyle\frown}}
\newcommand{\BZ}{{\mathbb Z}}  %
\newcommand{\ee}[3]{\operatorname{E}^{#1}_{#2,#3}}
\newcommand{\de}[3]{\operatorname{d}^{#1}_{#2,#3}}
\newcommand{\eec}[3]{\operatorname{E}_{#1}^{#2,#3}}
\newcommand{\dec}[3]{\operatorname{d}_{#1}^{#2,#3}}
\newcommand{\ov}{\overline}
\newcommand{\col}{\colon}
\newcommand{\wt}{\widetilde}

\newcommand{\bsh}{{\boldsymbol h}}
\newcommand{\omicron}{{o}}

\newcommand{\ges}{\geqslant}
\newcommand{\les}{\leqslant}

\newcommand{\id}{\operatorname{id}}
\newcommand{\rank}{\operatorname{rank}}
\newcommand{\Ker}{\operatorname{Ker}}
\newcommand{\res}[2]{\mathsf{F}_{#1}{(#2)}}
\newcommand{\Hom}{\operatorname{Hom}}
\newcommand{\End}{\operatorname{End}}
\newcommand{\Ext}{\operatorname{Ext}}
\newcommand{\oshriek}{{\vphantom{\ov A}}^{\text{\rm<}}}
\newcommand{\Coker}{\operatorname{Coker}}

\theoremstyle{remark}

\theoremstyle{plain}
\newtheorem{theorem}{Theorem}[section]

\newtheorem{proposition}[theorem]{Proposition}
\newtheorem{lemma}[theorem]{Lemma}
\newtheorem{corollary}[theorem]{Corollary}
\newtheorem*{main}{Theorem}

\theoremstyle{remark}
\newtheorem*{Claim}{Claim}

\newtheorem{remark}[theorem]{Remark}
\newtheorem{notation}[theorem]{Notation}
\newtheorem*{Notes}{Notes}

\swapnumbers
\theoremstyle{definition}
\newtheorem{chunk}[theorem]{}

\newenvironment{bfchunk}{\begin{chunk}\textbf}{\end{chunk}}

\numberwithin{equation}{theorem}

\theoremstyle{remark}

\begin{document}
\title[Contravariant Koszul duality]
{(Contravariant) Koszul duality for DG algebras}
\author[L.~L.~Avramov]{Luchezar L.~ Avramov}
\address{Department of Mathematics,University of Nebraska,
Lincoln, NE 68588,U.~S.~A.}
\email{avramov@math.unl.edu}

   \begin{abstract}
A DG algebras $A$ over a field $k$ with $\mathrm{H}(A)$ connected and 
$\mathrm{H}_{<0}(A)=0$ has a unique up to isomorphism DG module $K$ 
with $\mathrm{H}(K)\cong k$.  It is proved that if $\HH(A)$ is degreewise finite, then 
$\mathrm{RHom}_A(?,K)\col\mathsf{D^{df}_{+}}(A)^{\mathsf{op}}\equiv\mathsf{D_{df}^{+}}(\mathrm{RHom}_A(K,K))$
is an exact equivalence of derived categories of DG modules with degreewise finite-dimensional 
homology. It induces an equivalences of $\mathsf{D^{df}_{b}}(A)^{\mathsf{op}}$ 
and the category of perfect DG $\mathrm{RHom}_A(K,K))$ modules, and vice-versa.
Corresponding statements are proved also when $\mathrm{H}(A)$ is simply 
connected and $\mathrm{H}^{<0}(A)=0$.
  \end{abstract}

\maketitle

\tableofcontents

\section{Introduction}

\numberwithin{equation}{theorem}

This paper is concerned with certain subcategories of the derived category $\dcatdf{}{}{A}$ of left DG 
(differential graded) $A$-modules, when $A$ is a DG algebra over a field~$k$.  Quotations from the
introductions of two foundational papers will help explain the approach taken here and the parentheses 
in the title.  

In \emph{Koszul duality} \cite[p.\,317]{BGSc} Beilinson, Ginzburg and Schechtman write

\smallskip\noindent
{\footnotesize [\dots] we consider differential graded algebras that satisfy some natural conditions (we call them mixed algebras).
For any mixed algebra $A$ one defines the new mixed algebra $\check A$ called Koszul dual of $A$ (the cohomology 
of $\check A$ equals to $\Ext$'s of $A$ with simple coefficients).  The derived categories $\dcatdf{}{}A$, 
$\dcatdf{}{}{\check A}$ (of differential graded modules) are canonically equivalent (``Koszul duality'').}
  \smallskip
  
This is fine-tuned by Beilinson, Ginzburg and Soergel in \cite[p.\,477]{BGSo}:

\smallskip\noindent
{\footnotesize For a Koszul ring $A$ one might interpret $\mathsf{E}(A)$ as being $\rhom_A(k,k)$, and then the Koszul
duality functor $K$ is just the functor $\rhom_A(k,\mathrm{?})$ [\dots] More details on this point of view 
can be found in~\cite{Ke}\footnote{The reference has been redirected to the bibliography in the present paper.}.}
 \smallskip
 
The present paper is about a \emph{duality} realized by a (\emph{contravariant}) exact functor:
  
  \begin{main}
Assume that the following hold: $\HH_0(A)=k$, $\rank_k\HH_i(A)$ is finite for each $i$, and
  \[
\HH_{i}(A)=0 \text{ for } i<0
\quad\text{respectively}\quad
\HH_{i}(A)=0 \text{ for } i>0
\text{ and }
\HH_{-1}(A)=0
  \]

The DG algebra $A$ then has a DG module ${}_Ak$ with $\sum_i\rank\HH_i({}_Ak)=1$, which is
unique up to shift and isomorphism in $\dcatdf{}{}{A}$.  It defines an exact equivalence
  \[
\xymatrixcolsep{1.5pc}
\xymatrixrowsep{2pc} 
\xymatrix{
\rhom_A(?,{}_Ak)\col \dcatdf{hf}{ha}{A}\opp
\ar@{->}[r]^-{\equiv}
&\dcatdf{hf}{ha}{\mathsf{E}(A)}
}
  \]
where $\mathsf{E}(A)={\rhom_A({}_Ak,{}_Ak)}$, and further restricts to exact equivalences
   \[
\xymatrixcolsep{1.5pc}
\xymatrixrowsep{2pc} 
\xymatrix{
\dcatdf{hf}{hb}{A}\opp
\ar@{->}[r]^-{\equiv}
&\dcatdf{}{perf}{\mathsf{E}(A)}
}
\quad\text{and}\quad 
\xymatrix{
\dcatdf{}{perf}{A}\opp
\ar@{->}[r]^-{\equiv}
&\dcatdf{hf}{hb}{\mathsf{E}(A)}
}
  \]
   \end{main}

The full subcategories in the theorem are described as follows: $\dcatdf{hf}{ha}{A}$ consists of 
the DG modules $M$ with $\rank_k\HH_i(M)$ finite for each $i$ and $\HH_i(M)=0$ for $i\ll0$,
respectively, for $i\gg0$; the objects of $\dcatdf{hf}{hb}{A}$ have $\sum_i\rank_k\HH_i(M)$ finite; 
$\dcatdf{}{perf}A$ is the thick subcategory generated by~$A$.

We give a complete, largely self-contained proof of the theorem.  The existence and uniqueness 
of $k_A$ was first established by Dwyer, Greenlees, and Iyengar \cite{DGI}.

Although our result involves only derived categories of DG modules, it is closely related to a 
covariant equivalence of the category of DG modules over an algebra $A$ and that of DG 
comodules over a coalgebra $C$, linked to $A$ via an acyclic twisting map {$\tau\col C\to A$}.  
Envisioned by John Moore \cite{Mo} for applications to homotopy theory, \emph{Moore equivalence} 
was developed by Husemoller, Moore and Stasheff \cite{HMS}, and others, by utilizing 
E.~H.~Brown's \cite{Br} construction of twisted tensor products.  

The first half of the paper contains a succinct presentation of twisted tensor products in the special
case needed here and a record of the behavior of the basic constructions under vector space 
duality.  This approach circumvents a number of complications that arise when working 
with derived categories of DG comodules.  The proof of the theorem is given in sections 
\ref{MooreDuality} through~\ref{KoszulDuality}; the present version 
incorporates simplifications suggested by work of F\'elix, Halperin and Thomas \cite{FHT1} on 
the homology of fibrations.  In the last sections of the main text we discuss, with a view towards 
applications, two classes of algebras for which explicit computations are available.  Three 
appendices handle terminology and notation concerning DG (co)algebras and DG 
(co)modules, allowing for a largely self-contained exposition.  In particular, no prior exposure 
to coalgebras or twisted tensor products is assumed.

The motivation for this paper came from the joint work \cite{AJ}, where the theorem above is needed.  
A few related earlier results are discussed at the end of Section~\ref{KoszulDuality}.

I want to thank Alexander Berglund, Ragnar-Olaf Buchweitz, Srikanth Iyengar, and Sarah Witherspoon
for useful conversations at various stages of this work, the anonymous referee for a thorough reading 
of an earlier version, and the editors of the conference proceedings for their patience and tact.

\section{Cup products and cap products}
\label{S:Cup products and cap products}

\numberwithin{equation}{theorem}

This section contains constructions that underpin all the work in the paper. 

The classical operation of convolution of functions turns the complex of linear homomorphisms 
from a DG coalgebra to DG algebra into a DG algebra.  The latter comes equipped with natural 
actions on complexes of homomorphisms from DG comodules to DG modules and on tensor 
products of such objects.  We give a concise and complete account of these 
constructions.  Following the tradition in algebraic topology we use the names cup product and 
cap product for the resulting multiplicative structures and adopt the corresponding notation.

  \begin{notation}
    \label{ch:notation1}
In this paper $k$ denotes a fixed field.  All complexes are defined over~$k$.  References to $k$ are 
often suppressed from terminology and notation.  In particular, $\otimes$ refers to tensor 
products over $k$ and $\Hom$ to spaces of $k$-linear homomorphisms. 

Throughout the paper the following notation is in force:
  \medskip

\centerline{
  \begin{tabular}{lll}
$A$ is a DG algebra
&\qquad\qquad
& $C$ is a DG coalgebra
\\
$M$ is a left DG $A$-module
&& $X$ is a left DG $C$-comodule
\\
$N$ is a right DG $A$-module
&& $Y$ is a right DG $C$-comodule
  \end{tabular}
}
\medskip

\noindent
The relevant definitions are recalled in Appendices \ref{app:DGA} and \ref{app:DGC}.
  \end{notation}

  \begin{bfchunk}{Cup products.}
    \label{ch:cup}
Set $\Xi_{CA}=\Hom(C,A)$.

The \emph{cup product} of $\xi\in\Xi_{CA}$ and $\zeta\in\Hom(X,M)$ is the composed map  
  \[
\xi\scup\zeta\col X\xra{\psi^{CX}} C\otimes X\xra{\xi\otimes\zeta} 
A\otimes M\xra{\varphi^{AM}} M
  \]
Thus, if $\psi^{CX}(y)=\sum_i c_i\otimes x_i$, then we have the expression
  \begin{equation}
   \label{eq:cup}
(\xi\scup\zeta)(x)=\sum_i (-1)^{|\zeta||c_i|}\xi(c_i)\zeta(x_i)
  \end{equation}

  \begin{Claim}
    \label{ch:Xi}
Cup products for $X=C$ and $M=A$ turn $\Xi_{CA}$ into a DG algebra 
with unit~$\eta^A\varepsilon^C$, known as the \emph{convolution algebra}; we
write $\Xi$ when no ambiguity arises.
   \end{Claim}

The following string of equalities\footnote{In order to keep displays readable we write $|$ instead of 
$\otimes$.  Brackets are placed under those $\ub{\text{compositions of maps}}$  that are modified 
at the given step.  Thus, computations can be followed by checking for commutativity small 
diagrams involving only the selected terms.} shows that $\eta^A\varepsilon^C$ is a right unit for $\Xi$:
  \[
\ub{\xi\scup(\eta^A\varepsilon^C)}
=\varphi^{A}\ub{(\xi|\eta^A\varepsilon^C)}\psi^{C}
=\ub{\varphi^{A}(A|\eta^A)}(\xi|k)\ub{(C|\varepsilon^C)\psi^{C}}
=\ub{\id^A\xi\id^C}
=\xi
  \]

The remaining axioms for DG algebra follow from the next assertion: 

  \begin{Claim}
    \label{ch:cupXM}
Cup products turn $\Hom(X,M)$ into a left DG module over $\Xi_{CA}$.
   \end{Claim}

A computation similar to the one above shows that $\eta^A\varepsilon^C$ is a left unit
for cup products.  Their associativity results from the next string of equalities:
\begin{align*}
\ub{\tau\scup(\xi\scup\zeta)}
&=\varphi^{AM}\ub{(\tau|\xi\scup\zeta)}\psi^{CX}\\
&=\varphi^{AM}(\tau|M)\ub{(C|\xi\scup\zeta)}\psi^{CX}\\
&=\varphi^{AM}\ub{(\tau|M)(C|\varphi^{AM})}\ub{(C|\xi|\zeta)}\ub{(C|\psi^{CX})\psi^{CX}}\\
&=\ub{\varphi^{AM}(A|\varphi^{AM})}\ub{(\tau|A|M)(C|\xi|M)}\ub{(C|C|\zeta)(\psi^C|X)}\psi^{CX}\\
&=\varphi^{AM}\ub{(\varphi^A|M)(\tau|\xi|M)(\psi^C|M)}(C|\zeta)\psi^{CX}\\
&=\varphi^{AM}\ub{(\tau\scup\xi|M)(C|\zeta)}\psi^{CX}\\
&=\ub{\varphi^{AM}(\tau\scup\xi|\zeta)\psi^{CX}}\\
&=(\tau\scup\xi)\scup\zeta
  \end{align*}
Another calculation, where $H=\Hom(X,M)$, verifies the Leibniz rule:
 \begin{align*}
\ub{\dd^{H}}\ub{(\xi\scup\zeta)}
&=\ub{\dd^M\varphi^{AM}}(\xi|\zeta)\psi^{CX}-(-1)^{|\xi|+|\zeta|}\varphi^{AM}(\xi|\zeta)\ub{\psi^{CX}\dd^X}\\
&=\varphi^{AM}\ub{\dd^{A\otimes M}(\xi|\zeta)}\psi^{CX}-(-1)^{|\xi|+|\zeta|}\varphi^{AM}\ub{(\xi|\zeta)\dd^{C\otimes X}}\psi^{CX}\\
&=\ub{\varphi^{AM}(\dd^{\Xi}(\xi)|\zeta)\psi^{CX}}+(-1)^{|\xi|}\ub{\varphi^{AM}(\xi|\dd^{H}(\zeta))\psi^{CX}}\\
&=\dd^{\Xi}(\xi)\scup\zeta+(-1)^{|\xi|}\xi\scup\dd^{H}(\zeta)
  \end{align*}
 \end{bfchunk}

  \begin{bfchunk}{Cap products.}
    \label{ch:cap}
For each $\xi\in\Xi_{CA}$ form the composed map
  \[
\lambda^{\Xi(Y\otimes M)}(\xi)\col 
Y\otimes M\xra{\psi^{YC}\otimes M} Y\otimes C\otimes M\xra{Y\otimes \xi \otimes M} 
Y\otimes A\otimes M\xra{Y\otimes \varphi^{AM}} Y\otimes M
  \]
For $y\in Y$ and $m\in M$ the \emph{cap product} $\xi\scap(y\otimes m)$ is defined to be the image 
of $y\otimes m$ under the map $\lambda^{\Xi(Y\otimes M)}(\xi)$.  Thus, if $\psi^{YC}(y)=\sum_i y_i\otimes c_i$, then
  \begin{equation}
    \label{eq:cap}
\xi\scap(y\otimes m)=\sum_i (-1)^{|\xi||y_i|}y_i\otimes\xi(c_i)m
  \end{equation}

\begin{Claim}
    \label{ch:capYM}
Cap products turn $Y\otimes M$ into a left DG $\Xi_{CA}$-module.
 \end{Claim}

Indeed, set $\lambda=\lambda^{\Xi(Y\otimes M)}$.  As in \ref{ch:cupXM}, it is easy to check that $\lambda$ 
is a morphism of complexes and that $\lambda(\eta^A\varepsilon^C)=\id^{Y\otimes M}$ holds.  We complete 
the verification that $\lambda$ is a representation of $\Xi_{CA}$ in $Y\otimes M$, see \ref{ch:DGAmod}, by 
the following computation:
\begin{align*}
\ub{\lambda(\tau)}\ub{\lambda(\xi)} 
&=(Y|\varphi^{AM)}(Y|\tau|M)\ub{(\psi^{YC}|M)(Y|\varphi^{AM})}(Y|\xi|M)(\psi^{YC}|M)\\
&=(\varphi^{AM}|M)\ub{(Y|\tau|M)(Y|C|\varphi^{AM})}\ub{(\psi^{YC}|A|M)(Y|\xi|M)}(\psi^{YC}|M)\\
&=\ub{(\varphi^{AM}|M)(Y|A|\varphi^{AM})}\ub{(Y|\tau|A|M)(Y|C|\xi|M)}\ub{(\psi^{YC}|C|M)(\psi^{YC}|M)}\\
&=(\varphi^{AM}|M)\ub{(Y|\varphi^A|M)(Y|\tau|\xi|M)(Y|\psi^C|M)}(\psi^{YC}|M)\\
&=\ub{(Y|\varphi^{AM})(Y|\tau\scup\xi|M)(\psi^{YC}|M)}\\
&=\lambda(\tau\scup\xi)
  \end{align*}
 \end{bfchunk}

  \begin{bfchunk}{Opposite cup products.}
Set $\Xi_{CA}^{\mathsf o}=\Hom(C,A)$.

The \emph{opposite cup product} of $\xi\in\Xi_{CA}^{\mathsf o}$ and $\zeta\in\Hom(Y,N)$ is the composite map
  \[
\xi\ocup\zeta\col C\xra{\,\psi^{YC}\,} Y\otimes C\xra{\,(-1)^{|\zeta||\xi|}\zeta\otimes\xi\,} 
N\otimes A\xra{\,\varphi^{NA}\,} N 
  \]

For $Y=C$ and $N=A$ comparison with \ref{ch:cup} gives $\xi\ocup\zeta=(-1)^{|\xi||\zeta|}\zeta\scup\xi$, hence:

\begin{Claim}
    \label{ch:Xiop}
Opposite cup products for $Y=C$ and $N=A$ turn $\Xi_{CA}^{\mathsf o}$ into \emph{the opposite DG algebra} of
$\Xi_{CA}$ from \ref{ch:Xi}; we write $\Xi\avop$ when there is no ambiguity.
      \end{Claim} 

The next assertion is verified by computations similar to those in \ref{ch:cupXM}.

 \begin{Claim}
    \label{ch:cupYN}
Opposite cup products turn $\Hom(Y,N)$ into a left DG $\Xi_{CA}^{\mathsf o}$-module.
      \end{Claim} 
      \end{bfchunk} 

  \begin{bfchunk}{Opposite cap products.}
    \label{ch:ocap}
For each $\xi\in\Xi_{CA}^{\mathsf{o}}$ form the composed map
  \[
\lambda^{\Xi\avop(N\otimes X)}(\xi)\col 
N\otimes X\xra{N\otimes\psi^{CX}} N\otimes C\otimes X\xra{N\otimes\xi\otimes X} 
N\otimes A\otimes X\xra{\varphi^{NA}\otimes X} N\otimes X
  \]
For $n\in N$ and $x\in X$ the \emph{opposite cap product} $\xi\ocap(y\otimes m)$ is defined to be
$\lambda^{\Xi\avop(N\otimes X)}(\xi)(n\otimes x)$.  Thus, if $\psi^{CX}(c)=\sum_i c_i\otimes x_i$, then
  \begin{equation}
    \label{eq:ocap}
\xi\ocap(n\otimes x)=\sum_i (-1)^{|\xi||n|}n\xi(c_i)\otimes x_i
  \end{equation}

  \begin{Claim}
    \label{ch:capNX}
Opposite cap products turn $N\otimes X$ into a left DG $\Xi_{CA}^{\mathsf o}$-module.
  \end{Claim}
The computations are parallel to those used for \ref{ch:capYM}, with $\lambda=\lambda^{\Xi\avop(N\otimes X)}$.  
             \end{bfchunk}

Convolution algebras are functorial in both arguments:

      \begin{bfchunk}{Morphisms.}
        \label{ch:morphisms}
Each pair $(\alpha,\gamma)$, where $\alpha\col A'\to A$ is a morphism of DG algebras and 
$\gamma\col C\to C'$ one of DG coalgebras, induces morphisms of DG algebras 
  \[
\Xi_{\gamma\alpha}\col \Xi_{C'A'}\to\Xi_{CA}
  \quad\text{and}\quad
\Xi_{\gamma\alpha}^{\mathsf o}\col\Xi_{C'A'}^{\mathsf o}\to\Xi_{CA}^{\mathsf o}
  \quad\text{given by}\quad
\xi\mapsto\alpha\xi\gamma 
  \] 
        \end{bfchunk}

   \begin{bfchunk}{Adjunction.}
    \label{ch:twistAdj0}
Given morphisms $\vartheta\col X\to C\otimes M$ of left DG $C$-comodules and $\theta\col A\otimes X\to M$ 
of left DG $A$-modules the composed maps
  \begin{align*}
\omega^{CA}(\vartheta)&\col 
A\otimes X
\xra{A\otimes \vartheta}
A\otimes C \otimes M
\xra{A\otimes\varepsilon^{C}\otimes M}
A\otimes M
\xra{\varphi^{AM}}
M
  \\
\omega^{AC}(\theta)&\col 
X
\xra{\psi^{CX}}
C\otimes X 
\xra{C\otimes\eta^{A}\otimes X}
C\otimes A\otimes X
\xra{C\otimes\theta}
C\otimes M
  \end{align*}
are morphisms of left DG $A$-modules and of left DG $C$-comodules, respectively.

Similarly, morphisms $\vartheta\col Y\to N\otimes C$ and $\theta\col Y\otimes A\to N$ yield morphisms
  \begin{align*}
\omega^{CA}(\vartheta)&\col 
Y\otimes A
\xra{\vartheta\otimes A}
N\otimes C \otimes A
\xra{N\otimes\varepsilon^{C}\otimes A}
A\otimes M
\xra{\varphi^{NA}}
N
  \\
\omega^{AC}(\theta)&\col 
Y
\xra{\psi^{YC}}
Y\otimes C
\xra{Y\otimes\eta^{A}\otimes C}
Y\otimes A\otimes C
\xra{\theta\otimes C}
N\otimes C
  \end{align*}
of right DG $A$-modules and of right DG $C$-comodules, respectively.
   \end{bfchunk}

In the next lemma, as elsewhere else in the paper, pairs of adjoint functors are displayed 
so that the left adjoint appears on top.

  \begin{lemma}
    \label{lem:omega}
The following maps are inverse isomorphisms of complexes 
  \begin{alignat}{2}
    \label{eq:twistAdjXM}
\xymatrixcolsep{2pc}
\xymatrixrowsep{2pc} 
\omega^{CA}&\col 
\xymatrix{
\Hom_C(X,C\otimes M) 
\ar@{->}[r]<1ex>_-{\cong}
&\Hom_A(A\otimes X, M)
\ar@{->}[l]<1.2ex>
}
&&:\!\omega^{AC}
\\
    \label{eq:twistAdjYN}
\omega^{CA}&\col 
\xymatrix{
\Hom_C(Y,N\otimes C) 
\ar@{->}[r]<1ex>_-{\cong}
& \Hom_A(Y\otimes A, N)
\ar@{->}[l]<1.2ex>
}
&&:\!\omega^{AC}
  \end{alignat}
They commute with the actions of $\Xi_{CA}$ and $\Xi_{CA}^{\mathsf o}$ in the
following sense:
  \begin{align*}
\omega^{CA}(\lambda^{\Xi(C\otimes M)}(\xi)\circ\vartheta)
&=
\omega^{CA}(\vartheta)\circ\lambda^{\Xi\avop(A\otimes X)}(\xi)
  \\
\omega^{AC}(\theta\circ\lambda^{\Xi\avop(A\otimes X)}(\xi))
&=
\lambda^{\Xi(C\otimes M)}(\xi)\circ\omega^{AC}(\theta)
  \\
\omega^{CA}(\lambda^{\Xi\avop(N\otimes C)}(\xi)\circ\vartheta)
&=
\omega^{CA}(\vartheta)\circ\lambda^{\Xi(Y\otimes A)}(\xi)
  \\
\omega^{AC}(\theta\circ\lambda^{\Xi(Y\otimes A)}(\xi))
&=
\lambda^{\Xi\avop(N\otimes C)}(\xi)\circ\omega^{AC}(\theta)
   \end{align*}
  \end{lemma}

  \begin{proof}
Indeed, $\omega^{AC}$ and $\omega^{CA}$ are morphisms of complexes because they are 
induced by morphisms of DG modules, respectively, of DG comodules.  The adjunction is
best seen by decomposing it into standard pieces, which for the first one are
  \[
\xymatrixcolsep{1.5pc}
\xymatrixrowsep{1pc} 
\xymatrix{
\Hom_A(A\otimes X, M)
\ar@{<-}[r]<1ex>_-{\cong}
&\Hom(X, M)
\ar@{<-}[l]<1.2ex>
\ar@{<-}[r]<1ex>_-{\cong}
&\Hom_C(X,C\otimes M) 
\ar@{<-}[l]<1.2ex>
}
  \]

The first commutation relation in the lemma above is verified as follows:
  \begin{align*}
\ub{\omega^{CA}(\vartheta)}\circ\ub{\lambda^{\Xi\avop(A\otimes C)}(\xi)}
&=\varphi^{AM}(A|\varepsilon^{C}|M)\ub{(A|\vartheta)(\varphi^{A}|X)}(A|\xi |X)(A|\psi^{CX})\\
&=\varphi^{AM}\ub{(A|\varepsilon^{C}|M)(\varphi^{A}|C|X)}\ub{(A|A|\vartheta)(A|\xi |X)}(A|\psi^{CX})\\
&=\ub{\varphi^{AM}(\varphi^{A}|M)}\ub{(A|A|\varepsilon^{C}|M)(A|\xi|C|M)}\ub{(A|C|\vartheta)(A|\psi^{CX})}\\
&=\varphi^{AM}(A|\varphi^{AM})(A|\xi|M)\ub{(A|C|\varepsilon^C|M)(A|\psi^{C(C\otimes M)})}(A|\vartheta)\\
&=\varphi^{AM}(A|\varphi^{AM})\ub{(A|\xi|M)(A|\varepsilon^C|C|M)}(A|\psi^{C}|M)(A|\vartheta)\\
&=\varphi^{AM}\ub{(A|\varphi^{AM})(A|\varepsilon^C|A|M)}(A|C|\xi|M)(A|\psi^{C}|M)(A|\vartheta)\\
&=\ub{\varphi^{AM}(A|\varepsilon^C|M)}\ub{(A|C|\varphi^{AM})(A|C|\xi|M)(A|\psi^{C}|M)(A|\vartheta)}\\
&=\omega^{CA}(\lambda^{\Xi(C\otimes M)}(\xi )\circ\vartheta)
  \end{align*}
Similar calculations establish the remaining relations. 
 \end{proof}
 
  \begin{Notes}
Only the action of specific elements of the convo\-lution algebra is needed to twist
tensor products, but the DG module structures introduced above help in computations
and demystify some formulas.   This point of view is emphasized by 
Huebschmann~\cite{Hu} and is taken up by Loday and Vallette in the recent book \cite{LV}.
  \end{Notes}
  
\section{Twisted tensor products}
\label{TwistedTensorProducts}

\numberwithin{equation}{theorem}

A twisting map, also called twisting cochain or twisting morphism, is an element of the 
convolution algebra that allows for \emph{uniform} and \emph{functorial} modifications
of the differentials of \emph{all} tensor products of DG modules with DG comodules.

The K\"unneth formula shows that ordinary tensor products preserve quasi-isomor\-phisms.
We are interested to know when twisted tensor products have a similar property.  As  
twists scramble the original ``good'' differential of a tensor product, the idea is to use those
that do not interfere too much with the differentials of one factor.  This can be expressed in 
terms of filtrations and analyzed by means of the associated spectral sequences.  To reach
conclusions one needs guarantees of convergence, which are provided by appropriate bounds on the factors.

  \begin{notation}
    \label{ch:notation2}
The notation from \ref{ch:notation1} is in force and we set $\Xi=\Xi_{CA}$; see \ref{ch:cup}.
  \medskip
  
A \emph{twisting map} is a $k$-linear homomorphism $\tau\col C\to A$ of degree $-1$, such that
 \begin{equation}
    \label{eq:twist}
  \dd^A\tau+\tau\dd^C=\varphi^A(\tau\otimes\tau)\psi^C 
  \end{equation}
In view of \ref{ch:Xi} and \ref{ch:Xiop}, this is equivalent to either one of the equalities
  \begin{equation}
    \label{eq:twistXi}
\dd^{\Xi}(\tau)=\tau\scup\tau
  \qquad\text{or}\qquad
\dd^{\Xi\avop}(\tau)=-\tau\ocup\tau 
  \end{equation}
That is, to either condition: $\tau$ is a twister in $\Xi$ or that $-\tau$ is one in $\Xi\avop$; see \ref{ch:twisters}(1).
     \end{notation}

  \begin{bfchunk}{Twisted tensor products.}
    \label{ch:twist}
Let $N\rotimes\tau X$ be the complex produced by the construction in \ref{ch:twisters}(2), applied to the left DG 
$\Xi\avop$-module $N\otimes X$ from \ref{ch:capNX} and a twisting map $\tau$; in detail: $(N\rotimes\tau X)\nat=(N\otimes X)\nat$ and
  \begin{equation}
    \label{eq:twistXi2}
\dd^{N\rotimes{\tau} X}=\dd^{N\otimes X}+\lambda^{\Xi\avop(N\otimes X)}(\tau)
  \end{equation}
It is clear that $A\rotimes{\tau} X$ is a left DG $A$-module with action $\varphi^{A(A\rotimes{\tau} X)}:=\varphi^A\otimes X$
and $N\rotimes{\tau} C$ is a right DG $C$-comodule with coaction $\psi^{(N\rotimes{\tau} C)C}:=N\otimes \psi^C$

On the other hand, let $Y\avlotimes\tau M$ denotes the complex obtained in \ref{ch:twisters}(2) 
from the left DG $\Xi$-module $Y\otimes M$, see \ref{ch:capYM}; it has $(Y\avlotimes\tau M)\nat=(Y\otimes M)\nat$ and
  \begin{equation}
    \label{eq:twistXi1}
\dd^{Y\avlotimes{\tau} M}=\dd^{Y\otimes M}-\lambda^{\Xi(Y\otimes M)}(\tau)
  \end{equation}
Now $Y\avlotimes{\tau} A$ is a right DG $A$-module with action $\varphi^{(Y\avlotimes{\tau} A)A}:=Y\otimes\varphi^A$
and $C\avlotimes{\tau} M$ is a left DG $C$-comodule with coaction $\psi^{C(C\avlotimes{\tau} M)}:=\psi^C\otimes M$.

Both constructions are functorial and go by the name of \emph{twisted tensor products}.  

We always place the subscript $\tau$ on the side of the comodule argument.
     \end{bfchunk}

    \begin{bfchunk}{Associativity.}
    \label{prop:ass}
The $k$-linear endomorphism of $(N\otimes C\otimes M)\nat$ given by
  \[
\dd^{N\rotimes\tau C\avlotimes\tau M}
=\dd^{N\otimes C\otimes M}+\lambda^{\Xi\avop(N\otimes C)}(\tau)\otimes M-N\otimes\lambda^{\Xi(C\otimes M)}(\tau)
  \]
defines a complex $N\rotimes\tau C\avlotimes\tau M$ and the natural maps are isomorphisms of complexes
  \begin{equation}
     \label{eq:assoTwistsNCM}
N\rotimes\tau(C\avlotimes\tau M)\cong N\rotimes\tau C\avlotimes\tau M\cong(N\rotimes\tau C)\avlotimes\tau M 
  \end{equation}
  
The $k$-linear endomorphism of $(Y\otimes A\otimes X)\nat$ given by
  \[
\dd^{Y\avlotimes\tau A\rotimes\tau X}
=\dd^{Y\otimes A\otimes X}-\lambda^{\Xi(Y\otimes A)}(\tau)\otimes X+Y\otimes\lambda^{\Xi\avop(A\otimes X)}(\tau)
  \]
defines a complex $Y\avlotimes\tau A\rotimes\tau X$ and the natural maps are isomorphisms of complexes
  \begin{equation}
     \label{eq:assoTwistsYAX}
Y\avlotimes\tau(A\rotimes\tau X)\cong Y\avlotimes\tau A\rotimes\tau X\cong(Y\avlotimes\tau A)\rotimes\tau X
  \end{equation}

We use the isomorphisms above to identify the complexes involved.
   \end{bfchunk}

   \begin{bfchunk}{Adjointness.}
    \label{ch:twistAdj}
There are pairs of mutually inverse natural isomorphisms
 \begin{alignat}{2}
\xymatrixcolsep{2pc}
\xymatrixrowsep{2pc} 
\omega^{AC}&
\col \xymatrix{
\Hom_A(A\rotimes{\tau} X, M)
\ar@{<-}[r]<1ex>_-{\cong}
&\Hom_C(X,C\avlotimes{\tau} M) 
\ar@{<-}[l]<1.2ex>
}
&&:\!\omega^{CA}
\\
\omega^{AC}&\col 
\ \xymatrix{
\Hom_A(Y\rotimes{\tau} A, N)
\ar@{<-}[r]<1ex>_-{\cong}
&\Hom_C(Y,N\avlotimes{\tau} C)
\ar@{<-}[l]<1.2ex>
}
&&:\!\omega^{CA}
  \end{alignat}
of complexes, given by the maps defined in \ref{ch:twistAdj0}: This follows from the expressions 
for the differentials in \ref{ch:twist} and the commutation formulas in Lemma \ref{lem:omega}. 
   \end{bfchunk}

   \begin{bfchunk}{Naturality.}\!
     \label{ch:funTensor}
When $\gamma\col C\to C'$ is a morphism of DG coalgebras and $\tau'\col C'\to A$ is 
a twisting map, \ref{ch:morphisms} implies that the map $\tau=\tau'\gamma\col C\to A$ is twisting.  
By using \eqref{eq:twistXi2} and \eqref{eq:twistXi1} it is easy to see that the assignments 
$c\otimes a\mapsto \gamma(c)\otimes a$ and $a\otimes c\mapsto a\otimes\gamma(c)$
define  $\gamma$-equivariant morphisms of DG $A$-modules:
  \begin{equation}
    \label{eq:funGamma}
\gamma\otimes A\col C\avlotimes{\tau}A\to C'\avlotimes{\tau'}A
  \quad\text{and}\quad
A\otimes\gamma\col A\rotimes{\tau}C\to A\rotimes{\tau'}C'
    \end{equation}

Similarly, when $\alpha\col A'\to A$ is a morphism of DG algebras and $\tau'\col C\to A'$ is 
a twisting map the map $\tau=\alpha\tau'\col C\to A$ is twisting and $c\otimes a'\mapsto c\otimes\alpha(a')$ 
and $a'\otimes c\mapsto \alpha(a')\otimes c$ define  $\alpha$-equivariant morphisms of DG $C$-comodules:
  \begin{equation}
    \label{eq:funAlpha}
C\otimes\alpha\col C\avlotimes{\tau'}A'\to C\avlotimes{\tau}A
  \quad\text{and}\quad
\alpha\otimes C\col A'\rotimes{\tau'}C\to A\rotimes{\tau}C
  \end{equation}
    \end{bfchunk}

\begin{proposition}
      \label{prop:ACCA}
Assume that $A$ is augmented, $C$ is coaugmented, and
  \[
\text{\rm(p)}\quad
\ov A_{\les0}=0=\ov C_{\les1} 
\qquad\text{respectively}\qquad
\text{\rm(n)}\quad
\ov A_{\ges-1}=0=\ov C_{\ges0}
  \]
  \begin{enumerate}[\rm(1)]
  \item
If $\gamma\col C\to C'$ is a quasi-isomorphism of coaugmented DG coalgebras with $\ov C'{}_{\les-1}=0$, respectively, 
$\ov C{}'_{\ges0}=0$, then $A\otimes\gamma$ and $\gamma\otimes A$ in \eqref{eq:funGamma} are homotopy equivalences.
  \item
If $\alpha\col A'\to A$ is a quasi-isomorphism of augmented DG algebras with $\ov A{}'_{\les0}=0$, respectively, 
$\ov A{}'_{\ges-1}=0$, then $\alpha\otimes C$ and $C\otimes\alpha$ in \eqref{eq:funAlpha} are quasi-isomorphisms.
 \end{enumerate}
     \end{proposition}

  \begin{proof}
The arguments for the four assertions are parallel.  We present one of them.

For $c\in C_p$ with $\ov\psi{}^C(c)=\sum_{i} c_i\otimes c'_i$ and $a\in A_q$ 
formulas \eqref{eq:twistXi2} and \eqref{eq:cap} express $\dd^{C\avlotimes{\tau} A}(c\otimes a)$ as a 
sum of the following terms:
  \begin{alignat}{2}
    \label{eq:struc1}
{\phantom{ + }}\sum_{|c_i|=j}(-1)^{|c_i|}c\otimes\tau_j(c'_i)a &\in C_{p-j}\otimes A_{q+j-1}
\quad\text{for}\quad j\ne0,1
  \\
    \label{eq:struc2}
c\otimes\dd^A_q(a) + (-1)^{p}c\otimes\tau_0(1)a &\in C_{p}\otimes A_{q-1}
\\
    \label{eq:struc3}
\dd^C_p(c)\otimes a - \sum_{|c_i|=1}(-1)^{p}c_i\otimes\tau_1(c'_i)a &\in C_{p-1}\otimes A_q
  \end{alignat}

When (p) holds we have $\tau'_{\les1}=0$, so \eqref{eq:struc1} shows that $(C'_{\les p}\otimes A)_{p\in\BZ}$ 
is an increasing filtration of $C'\avlotimes{\tau'}A$ by subcomplexes.  It defines a spectral sequence 
$(\de rpq\col\ee rpq\to\ee r{p-r}{q+r-1})_{r\ges0}$ that lies in the first quadrant and converges to $\HH(C'\avlotimes{\tau'}A)$ from
$\ee 0pq=C'_{p}\otimes A_q$.  Now \eqref{eq:struc2} gives $\de 0pq=C'_p \otimes(-1)^p\dd^A_q$, 
so $\ee 1pq=C'_p\otimes\HH_q(A)$. From \eqref{eq:struc3} we get $\de 1pq=\dd^{C'}_p\otimes\HH_q(A)$, 
so $\ee 2pq=\HH_p(C')\otimes\HH_q(A)$. 

A similar spectral sequence with second page equal to $\HH(C')\otimes\HH(A)$ converges to 
$\HH(C\avlotimes{\tau'\gamma}A)$.  The map $\gamma\otimes A$ induces a morphism from
the latter sequence to the former one.  On the second page it is the isomorphism $\HH(\gamma)\otimes\HH(A)$, 
so $\HH(\gamma\otimes A)$ is bijective by the classical comparison theorem for spectral sequences.

By Lemma \ref{lem:semifree} below, both $C\avlotimes{\tau'\gamma}A$ and $C'\otimes A$ are semifree
over $A$, so $\gamma\otimes A$ is a homotopy equivalence of right DG $A$-modules; see \ref{ch:DGMhp}.
 
When (n) holds we have $\tau'_{\ges0}=0$, so $(C\otimes(A_{\les q}))_{q\in\BZ}$ is a decreasing filtration by 
subcomplexes.  Its spectral sequence $(\dec rpq\col\eec rpq\to\eec r{p+r}{q-r+1})_{r\ges0}$ converges to 
$\HH(C\avlotimes{\tau}A)$ from $\eec 0pq=C_{-q}\otimes A_{-p}$.  The proof then concludes as above.
  \end{proof}

The next result introduces into the picture semifree DG modules; see \ref{ch:DGMhp}.

 \begin{lemma}
      \label{lem:semifree}
Assume that either $A_{\les-1}=0$, or $A$ is augmented and $\ov A_{\ges-1}=0$.

When $M$, respectively, $N$ is adequate for $A$, it is semifree if and only if its underlying graded 
$A\nat$-module is free.
    \end{lemma}

  \begin{proof} 
``Only if'' is evident.  For the converse we present the arguments for right modules.
We may write $N\nat$ as $V\otimes A\nat$ for some $k$-vector space $V$ that is adequate for $A$.  
By the Leibniz rule, $\dd^N$ is determined by its restriction on $V\otimes k$.

When $A_{\les-1}=0$ the graded submodules $F^p:=V_{\les p}\otimes A$ are subcomplexes for
degree reasons, and $F^{p}/F^{p-1}\cong(\shift^pV^p)\otimes A$ holds as right DG $A$-modules.

Assume now than $A$ is augmented and $\ov A_{\ges-1}=0$.  For $v\in V_j$ we then have
  \[
\dd^N(v\otimes 1)=v'\otimes 1+\sum_{|v_i|\ges j+1}v_i\otimes a_i  
\quad\text{with}\quad v'\in V_{j-1}
  \]
whence $\dd^N(v'\otimes1)\in V_{\ges j}\otimes A$.   It follows that the sequence of inclusions
  \[
\cdots\subseteq(V_{\ges p+1}\oplus\shift^{p}V'_{p})\otimes A \subseteq
V_{\ges p}\otimes A \subseteq
(V_{\ges p}\oplus\shift^{p-1}V'_{p-1})\otimes A \subseteq\cdots
  \]
where $V'_j=\{v\in V_j \mid\dd^N(v\otimes1)\in V_{>j}\otimes A\}$ is a semifree filtration of $N$.
 \end{proof}

\section{Acyclic twisting maps}
  \label{AcyclicTwistingMaps}

Here we discuss those twisting maps that are most important for applications.

  \begin{notation}
    \label{ch:notation4}
The notation from \ref{ch:notation1} stays in force, with the following additions.

We assume that $\varepsilon^A\col A\to k$ is an augmented DG algebra and that $\eta^C\col k\to C$ 
is a coaugmented DG coalgebra, see \ref{ch:AugC} and \ref{ch:AugA}, respectively; thus, $k$ is a left and right
DG $A$-module and a left and right DG $C$-comodule.  

We let $\tau\col C\to A$ denote a twisting map, see \ref{ch:notation2}, such that $\varepsilon^A\tau=0=\tau\eta^C$.  
  \end{notation}

   \begin{bfchunk}{Comparison maps.}\!
      \label{ch:comparison}
The units and counts of the adjunctions in \ref{ch:twistAdj} give maps
 \begin{align}
    \label{eq:quismAugACM}
\varepsilon^{ACM}&\col 
A\rotimes{\tau}C\avlotimes{\tau} M
\xra{A\otimes\varepsilon^{C}\otimes M}
A\otimes k\otimes M
=A\otimes M
\xra{\,\varphi^{AM}\,}
M
  \\
    \label{eq:quismAugCAX}
\eta^{CAX}&\col 
X\xra{\,\psi^{CX}\,}C\otimes X=C\otimes k\otimes X
\xra{C\otimes\eta^{A}\otimes X}C\avlotimes{\tau} A\rotimes{\tau} X
  \\
    \label{eq:quismAugNCA}
\varepsilon^{NCA}&\col 
N\rotimes{\tau}C\avlotimes{\tau} A
\xra{N\otimes\varepsilon^{C}\otimes A}
N\otimes k\otimes A
=N\otimes A
\xra{\,\varphi^{NA}\,}
N
  \\
    \label{eq:quismAugYAC}
\eta^{YAC}&\col 
Y\xra{\,\psi^{YC}\,}Y\otimes C=Y\otimes k\otimes C
\xra{Y\otimes\eta^{A}\otimes C}Y\avlotimes{\tau} A\rotimes{\tau} C
  \end{align}
which---in the order displayed---are morphisms of left DG $A$-modules, left DG $C$-comodules, 
right DG $A$-modules, and rigt DG $C$-comodules, respectively.
   \end{bfchunk}

   \begin{bfchunk}{Acyclicity.}
      \label{ch:acyclicity}
The twisting map $\tau$ is said to be \emph{acyclic} if the map
\begin{alignat}{2}
    \label{eq:acyclicity1}
&\varepsilon^{ACA} 
  \quad\text{given by}\quad
\varepsilon^{ACA}(a\otimes c\otimes a')=\varepsilon^C(c)aa'
   \end{alignat}
is a quasi-isomorphism. 
     \end{bfchunk}

Classical constructions provide \emph{universal} acyclic twisting maps.  The concepts 
used in their descriptions are defined in \ref{ch:tena} and \ref{ch:tenc}.  Detailed presentations 
abound; see e.g.\ \cite[II]{HMS}, \cite[Ch.\ 19]{FHT2}, \cite[Ch.\ 10]{Ne}, or \cite[Ch.\ 1 and 2]{LV}.

\begin{bfchunk}{Bar construction.}
    \label{ch:Bar}
When $\varepsilon^A$ is an augmented DG algebra the \emph{bar construction} $\babar A$ 
is the coaugmented DG coalgebra, described as follows.
  \begin{enumerate}[\rm(1)]
 \item
The underlying coalgebra $\babar A\nat$ is the tensor coalgebra $\ten c{\shift{\ov A}}\nat$; the tradition is to write
$[a_1|\cdots| a_p]$ for $(\susp{a_1})\otimes\cdots\otimes(\susp{a_p})$ and $1$ for $[\ ]$ in $V^{\otimes0}$.
  \item
The differential $\dd^{\babar A}$ is the unique coderivation of $\ten cV$ with $\pi\dd^{\babar A}(V^{\otimes p})=0$ 
for $p\ne 1,2$, $\pi\dd^{\babar A}([a_1|a_2])=(-1)^{|a_1|}[a_1a_2]$, and $\pi\dd^{\babar A}([a])=-[\dd^A(a)]$.
    \end{enumerate}

The bar construction has the following properties.
  \begin{enumerate}[\rm(1)]
 \item[\rm(3)]
The map $\tau^A\col\babar A\to A$ with $\tau^A([a])=a$ and $\tau^A(V^{\otimes p})=0$ for $p\ne1$ is twisting.
  \item[\rm(4)]
Both $\varepsilon^{A(\babar A)M}$  from \eqref{eq:quismAugACM} and $\varepsilon^{N(\babar A)A}$ from \eqref{eq:quismAugNCA} are quasi-isomorphisms.
  \item[\rm(5)]
When $C$ is cocomplete, see \ref{ch:AugC}, there is a unique morphism of DG coalgebras $\gamma^\tau\col C\to\babar A$ 
satisfying $\tau=\tau^A\gamma^\tau$.
    \end{enumerate}
    \end{bfchunk}

\begin{bfchunk}{Cobar construction.}
    \label{ch:Cobar}
For a coaugmented DG coalgebra $\eta^C$ the \emph{cobar construction} $\cobar C$ is 
the augmented DG algebra described as follows.
  \begin{enumerate}[\rm(1)]
 \item
The algebra $\cobar C\nat$ is the tensor algebra $\ten a{\shift^{-1}{\ov C}}\nat$; the tradition is to write
 $[c_1|\cdots| c_p]$ for $(\susp^{-1}{c_1})\otimes\cdots\otimes(\susp^{-1}{c_p})$ and $1$ for $[\ ]$ in $V^{\otimes0}$.
  \item
The differential $\dd^{\cobar C}$ is the unique derivation of $\cobar C$ satisfying the condition
$\dd^{\cobar C}([c])=\sum_i(-1)^{|c_i|}[c_i|c'_{i}]-[\dd^C(c)]$, where $\ov{\psi}{}^C(c)=\sum_{i}c_{i}\otimes c'_{i}$.
 \end{enumerate}

The cobar construction has the following properties.
  \begin{enumerate}[\rm(1)]
 \item[\rm(3)]
The map $\tau^C\col C\to \cobar C$ with $\tau^C(c)=[c-\eta^C\varepsilon^C(c)]$ is twisting.
 \item[\rm(4)]
Both $\eta^{C(\cobar C)X}$ from \eqref{eq:quismAugCAX} $\eta^{Y(\cobar C)C}$ from \eqref{eq:quismAugYAC} are quasi-isomorphisms.
 \item[\rm(5)]
There is a unique morphism of DG algebras $\alpha^{\tau}\col \cobar C\to A$ satisfying $\tau=\alpha^{\tau}\tau^C$.
 \end{enumerate}
    \end{bfchunk}

  \begin{bfchunk}{(Co)augmentations.}
      \label{ch:Aug}
The morphisms $\varepsilon^A$ and $\eta^C$ turn $k$ into a left and right DG $A$-module and into a left and right 
DG $C$-comodule, respectively.  We have
  \begin{alignat}{3}
    \label{eq:quismAugAkkC}
Y\avlotimes{\tau}k&=Y \quad k\avlotimes{\tau}M=M
  &\quad&\text{and}\quad
&k\rotimes{\tau}X&=X \quad N\rotimes{\tau}k= N  
\intertext{with first equality from \eqref{eq:twistXi2} and \eqref{eq:cap}, etc.  Thus, there are equalities}
    \label{eq:quismAugAC}
\varepsilon^{A}\otimes\varepsilon^{C}&=\varepsilon^{ACk}\col A\rotimes{\tau}C\to k
  &\ &\text{and}\ 
&\varepsilon^{C}\otimes\varepsilon^{A}&=\varepsilon^{kCA}\col C\avlotimes{\tau}A\to k
  \\
    \label{eq:quismAugCA}
\eta^{A}\otimes\eta^{C}&=\eta^{ACk}\col k\to A\rotimes{\tau}C
  &\quad&\text{and}\quad
&\eta^{C}\otimes\eta^{A}&=\eta^{kCA}\col k\to C\avlotimes{\tau}A
  \end{alignat}
of morphisms of DG $A$-modules in \eqref{eq:quismAugAC} and of DG $C$-comodules in \eqref{eq:quismAugCA}.
        \end{bfchunk}

Most applications of twisting maps are based on the following criterion.

    \begin{theorem}
    \label{thm:quismAcy}
Let $A$ be an augmented DG algebra and $C$ a coaugmented DG  coalgebra such that 
$\ov A:=\Ker(\varepsilon^A)$ and $\ov C:=\Ker(\varepsilon^C)$ satisfy the conditions
  \[
\text{\rm(p)}\quad
\ov A_{\les0}=0=\ov C_{\les1} 
\qquad\text{respectively}\qquad
\text{\rm(n)}\quad
\ov A_{\ges-1}=0=\ov C_{\ges0}
  \]
  
A twisting map $\tau\col C\to A$ is acyclic if (respectively, only if) the maps described in one
(respectively, in all) of the following items are quasi-isomorphisms.
 \begin{enumerate}[\quad\rm(i)]
    \item
$\varepsilon^{ACM}\col A\rotimes{\tau}C\avlotimes{\tau}M\to M$ from \eqref{eq:quismAugACM} 
for all left DG modules~$M$.
    \item
$\varepsilon^{AC}\col A\rotimes{\tau}C\to k$ from \eqref{eq:quismAugAC}.
    \item
$\gamma^\tau\col C\to\babar A$ from \emph{\ref{ch:Bar}(5)} 
($\gamma^\tau$ exists because $C$ is cocomplete; see \emph{\ref{ch:AugC}}).
    \item[\rm(iv)]
$\eta^{CAX}\col X\to C\avlotimes{\tau}A\rotimes{\tau}X$ from \eqref{eq:quismAugCAX} 
for all left DG comodules~$X.{}$
    \item[\rm(v)]
$\eta^{CA}\col k\to C\avlotimes{\tau}A$ from \eqref{eq:quismAugCA}.
    \item[\rm(vi)]
$\alpha^{\tau}\col \cobar C\to A$ from \emph{\ref{ch:Cobar}(5)}.
  \end{enumerate}
  \end{theorem}
 
  \begin{proof}
Here we use the notation (y)$\implies$(z) as shorthand for the assertion: 

``If the map in (y) is a quasi-isomorphism, then so is the map in (z).''

\eqref{eq:acyclicity1}$\implies$(i).
Since $A\rotimes{\tau}C\avlotimes{\tau}A$ is semifree as a right 
DG $A$-module by Lemma \ref{lem:semifree}, and $\varepsilon^{ACA}$ is a morphism of 
right DG $A$-modules, it is a homotopy equivalence.  The following commutative diagram
then shows that $\varepsilon^{ACM}$ is a quasi-isomorphism:
   \[
\xymatrixcolsep{4pc}
\xymatrixrowsep{2pc} 
\xymatrix{
(A\rotimes{\tau}C\avlotimes{\tau}A)\otimes_A M
\ar@{->}[r]^-{\varepsilon^{ACA}\otimes_AM}_-{\simeq}
\ar@{->}[d]^-{\cong}
&A\otimes_AM
\ar@{->}[d]_-{\cong}
\\
A\rotimes{\tau}C\avlotimes{\tau}M
\ar@{->}[r]^-{\varepsilon^{ACM}}
&M
}
  \]

(i)$\implies$(ii).
The map $\varepsilon^{AC}$ is a quasi-isomorphism as it equals $\varepsilon^{ACk}$ by \eqref{eq:quismAugAC}.

(ii)$\implies$(iii).
As $\HH(A\rotimes{\tau}C)\cong k$ holds by hypothesis and $\HH(A\rotimes{\tau}\babar A)\cong k$ by 
\ref{ch:Bar}(4), the morphism $A\otimes\gamma^\tau\col A\rotimes{\tau}C\to A\rotimes{\tau}\babar A$
is a quasi-isomorphism.  Its source and target are semifree by Lemma \ref{lem:semifree}, so it yields the 
quasi-isomorphism in the following commutative diagram, where the equalities come from \eqref{eq:quismAugAkkC}:
   \[
\xymatrixcolsep{2pc}
\xymatrixrowsep{2pc} 
\xymatrix{
k\otimes_A(A\rotimes{\tau}C)
\ar@{->}[d]_-{k\otimes_A(A\otimes\gamma^\tau)}^-{\simeq}
\ar@{->}[r]^-{\cong}
&k\rotimes{\tau}C
\ar@{->}[d]_{k\otimes\gamma^\tau}
\ar@{=}[r]
&C
\ar@{->}[d]_-{\gamma^\tau}
\\
k\otimes_A(A\rotimes{\tau^A}\babar A)
\ar@{->}[r]^-{\cong}
&k\rotimes{\tau^A}\babar A
\ar@{=}[r]
&\babar A
}
  \]

(iii)$\implies$\eqref{eq:acyclicity1}.
We have $\tau=\tau^A\gamma^\tau$ by \ref{ch:Bar}(5).  Proposition \ref{prop:ACCA}(1) shows that the map
$A\otimes\gamma^\tau\col A\rotimes{\tau}C\to A\rotimes{\tau^A}\babar A$ is a quasi-isomorphism,
and the Proposition \ref{prop:ACCA}(2) shows that so is 
$A\otimes\gamma^\tau\otimes A\col A\rotimes{\tau}C\avlotimes{\tau}A\to A\rotimes{\tau^A}(\babar A)\avlotimes{\tau^A}A$.
Since $\varepsilon^{ACA}=\varepsilon^{A(\babar A)A}(A\otimes\gamma^\tau\otimes A)$
holds by \ref{ch:Bar}(5) and $\varepsilon^{A(\babar A)A}$ is a quasi-isomorphism by \ref{ch:Bar}(4),
we see that $\varepsilon^{ACA}$ is a quasi-isomorphism.

\eqref{eq:acyclicity1}$\implies$(iv)$\implies$(v)$\implies$(vi)$\implies$\eqref{eq:acyclicity1}
are proved by similar arguments.
  \end{proof}

   \begin{corollary}
    \label{cor:bar-cobar}
For each augmented DG algebra $A$ with $\ov A_{\les0}=0$ or $\ov A_{\ges-1}=0$ there is a canonical surjective 
quasi-isomorphism of DG algebras $\alpha^{A}\col\cobar{\babar A}\to A$.

For each coaugmented DG coalgebra $C$ with $\ov C_{\les1}$ or $\ov C_{\ges0}=0$ there is a canonical injective 
quasi-isomorphism of DG coalgebras $\gamma^{C}\col C\to\babar{\cobar C}$.
    \end{corollary}

  \begin{proof}
Given $A$, set $\alpha^A:=\alpha^{\tau^A}$.  Given $C$, set $\gamma^C:=\gamma^{\tau^C}$.
  \end{proof}

 \begin{bfchunk}{Natural resolutions.}
    \label{ch:naturalRes}
For every left DG module $M$ that is adequate for $A$ Theorem \ref{thm:quismAcy}  yields a 
natural quasi-isomorphism $\varepsilon^{ACM}\col A\rotimes{\tau}C\avlotimes{\tau}M\to M$.  
Since $A\rotimes{\tau}C\avlotimes{\tau}M$ is semifree by Lemma \ref{lem:semifree}, 
this is a \emph{functorial} semifree resolution.
  \end{bfchunk}

The universal constructions are natural and often preserve quasi-isomorphisms:

  \begin{bfchunk}{Naturality.}
  \label{ch:BarCobarNat}
When $\alpha\col A'\to A$ is a morphism of augmented DG algebras the assignment 
$[a'_1|\cdots|a'_p]\mapsto[\alpha(a'_1)|\cdots|\alpha(a'_p)]$ defines a morphism 
$\babar{\alpha}\col\babar{A'}\to\babar A$ of DG coalgebras, which is a quasi-isomorphism if $\alpha$ is one.

When $\gamma\col C\to C'$ is a morphism of coaugmented DG coalgebras the assignment 
$[c_1|\cdots|c_p]\mapsto[\gamma(c_1)|\cdots|\gamma(c_p)]$ 
defines a morphism $\cobar{\gamma}\col\cobar{C}\to\cobar C'$ of DG algebras; it is a quasi-isomorphism 
if $\gamma$ is one and $\ov C_{\les1}=0=\ov C{}'_{\les1}$ or $\ov C_{\ges0}=0=\ov C{}'_{\ges0}$

Indeed, the relevant (co)multiplicative properties follow directly from the definitions.  Filtering the 
corresponding constructions ``by the number of bars'' yields a morphism of spectral sequences 
with $\ee 2pq(\alpha)=(\HH(\alpha)^{\otimes p})_{p+q}$, respectively, 
$\ee 2pq(\gamma)=(\HH(\gamma)^{\otimes p})_{p+q}$.  When $\HH(\alpha)$ is an
isomorphism so is $\ee 2pq(\alpha)$, and then $\HH(\babar{\alpha})$ is one
because the filtrations in use are bounded below, ascending and exhaustive.  If $\gamma$ is a 
quasi-isomorphism, then so is $\ten a{\HH(\gamma)}$.  In this case the filtrations are bounded above,
descending and separated; the additional hypotheses on $C$ and $C'$ guarantee that the spectral 
sequences converge, so $\HH(\cobar\gamma)$ is an isomorphism. 
   \end{bfchunk}

  \begin{Notes}
The material in this section is classical.  Much of it goes back to the foundational papers 
of Gugenheim and Munkholm \cite{GM} and Husemoller, Moore, and Stasheff \cite{HMS}.  In
the first one algebras and coalgebras satisfy either condition (p) of Theorem \ref{thm:quismAcy},
or a condition weaker than (n).  The second one deals only with non-negatively graded objects 
under a hypothesis weaker than (p); this work is presented in detail in the recent book of 
Neisendorfer~\cite{Ne}.  The restrictions adopted here allow us to avoid the use of cotensor 
products and their derived functors.  
   \end{Notes}

  \section{Duals of DG coalgebras}
  \label{Duals of DG coalgebras}

Up to this point the treatment has been unbiased towards DG algebras or DG coalgebras.  
Now we start using vector space duality and this breaks the symmetry: Coalgebras 
and comodules turn into algebras and modules with actions on the same side, but 
restrictions on vector space dimensions are needed to go the other way.

First we look at situations that do not involve duals of algebras or modules.

   \begin{notation}
    \label{ch:notation5}
The notation in \ref{ch:notation1} is in force and $\mathrm{?}^*$ denotes vector space duality.
  \end{notation}

  \begin{bfchunk}{Duality.}
  \label{ch:dualActions}
Set $\eta^{C^*}=\varepsilon^C$ and, using $\varpi^{CC}$ from \eqref{eq:cx2}, form the composed map
  \[
\varphi^{C^*}\col C^*\otimes C^*\xra{\,\varpi^{CC}\,}(C\otimes C)^*\xra{\,(\psi^C)^*\,}C^*
  \]
These morphisms give $C^*$ a structure of DG algebra structure on $C^*$, and the maps 
  \begin{alignat*}{2}
\varphi^{C^*X^*}&\col C^*\otimes X^*\xra{\varpi^{CX}}(C\otimes X)^*\xra{(\psi^{CX})^*}X^*
  \\
\varphi^{Y^*C^*}&\col Y^*\otimes C^*\xra{\varpi^{YC}}(Y\otimes C)^*\xra{(\psi^{YC})^*}Y^*
   \end{alignat*}
turn $X^*$ into a left DG $C^*$-module and $Y^*$ into a right DG $C^*$-module.  

Indeed, the cup products from \ref{ch:Xi} with $A=k$ turns $C^*$ into a DG algebra $\Xi_{Ck}$ 
and $X^*$ into a left DG module over it.  Comparing definitions one sees that the corresponding 
identity maps are compatible with these structures; we identify the DG algebras $C^*$ and 
$\Xi_{Ck}$ and their left DG modules $X^*$ and $\Hom(X,k)$.  

On the other hand, the opposite cup products from \ref{ch:Xiop} turn $\Hom(Y,k)$ into a left DG  
over $\Xi_{Ck}^{\mathsf o}=C^*\avop$.  One checks that the associated right $C^*$-module is the 
module $Y^*$ from \ref{ch:dualActions}.  Once again, we make the corresponding identifications.
   \end{bfchunk}

      \begin{bfchunk}{Actions on tensor products.}
The equality $Y=Y\otimes k$ and \ref{ch:capYM} turn $Y$ into a left DG 
$C^*$-module, so $A\otimes C^*$ acts on $Y\otimes M$ from the left by the formula
  \begin{align}
    \label{eq:morphisms1}
(a\otimes\xi)(y\otimes m)&=(-1)^{|a|(|\xi|+|y|)}(\xi\scap y)\otimes(am)
   \end{align}

Similarly, the equality $X=k\otimes X$ and \ref{ch:capNX} turn $X$ into a left DG $C^*\avop$-module.  
Since $N$ is a left DG $A\avop$-module for the action $a\cdot n=(-1)^{|a||n|}na$, the formula
  \begin{align}
    \label{eq:morphisms2}
(a\otimes\xi)\cdot(n\otimes x)&=(-1)^{(|a|+|\xi|)|n|}(na)\otimes(\xi\ocap x)
  \end{align}
defines a left action of $A\avop\otimes C^*\avop$ on $N\otimes X$.
       \end{bfchunk}

  \begin{proposition}
        \label{prop:morphisms}
The assignment $a\otimes\xi\mapsto(c\mapsto\xi(c)a)$ defines injective morphisms 
  \[
\sigma^{AC^*}\col A\otimes C^*\to\Xi_{CA}
  \quad\text{and}\quad
\sigma^{A\avop C^*\avop}\col A\avop\otimes C^*\avop\to\Xi_{CA}^{\mathsf o}
  \]
of DG algebras; they are bijective if $A$ is finite, or if $C$ is finite, or if one of $A$ and $C$ is
degreewise finite and $A_{\ll0}=0=C_{\gg0}$ or $A_{\gg0}=0=C_{\ll0}$ holds.

The maps $\id^{Y\otimes M}$ and $\id^{N\otimes X}$ are equivariant over $\sigma^{AC^*}$
and $\sigma^{A\avop C^*\avop}$, respectively.
        \end{proposition}

  \begin{proof}
Arguing as in \ref{ch:cxDual}(3) we see that $\sigma^{AC^*}$ is an injective morphism of complexes, and 
is bijective as stated.  From $\sigma^{AC^*}(1\otimes\varepsilon^C)(c)=\varepsilon^C(c)1=\eta^A\varepsilon^C(c)$ 
we see that $\sigma^{AC^*}$ maps the unit of $A\otimes C^*$ to that of $\Xi_{CA}$.  To verify that $\sigma^{AC^*}$ 
preserves products we write $\psi^C(c)$ in the form $\sum_i c_i\otimes c'_i$ and use the string of equalities
  \begin{align*}
\sigma^{AC^*}((a\otimes\xi)(b\otimes\zeta))(c)
&=(-1)^{|\xi||b|}((\xi\scup\zeta)(c))ab\\
&=(-1)^{|\xi||b|}\sum_i(-1)^{|\zeta||c_i|}\xi(c_i)\zeta(c'_i)ab\\
&=\sum_{i}(-1)^{(|b|+|\zeta|)|c_i|}\xi(c_i)a\zeta(c'_i)b\\
&=\sum_{i}(-1)^{(|b|+|\zeta|)|c_i|}(\sigma^{AC^*}(a\otimes\xi)(c_i))(\sigma^{AC^*}(b\otimes\zeta)(c'_i))\\
&=(\sigma^{AC^*}(a\otimes\xi)\scup\sigma^{AC^*}(b\otimes\zeta))(c)
 \end{align*}
where the second and fifth ones come from formula \eqref{eq:cup}, while the third one reflects the 
observation that $\xi(c_i)$ is is in $k$, and so is zero when $|c_i|\ne-|\xi|$.  

The equality $(A\otimes C^*)\avop=A\avop\otimes C^*\avop$ now
shows that $\sigma^{A\avop C^*\avop}$ is a morphism of DG algebras. 
In the next computation the first and fourth equalities come from \eqref{eq:ocap}, the second 
equality from \eqref{eq:morphisms2}, and the third one holds because $\xi(c_i)$ is in $k$:
  \begin{align*}
(\sigma^{A\avop C^*\avop}(a\otimes\xi))\ocap(n\otimes x)
&=\sum_i(-1)^{(|a|+|\xi|)|n|}n\sigma^{A\avop C^*\avop}(a\otimes\xi)(c_i)\otimes x_i\\
&=\sum_i(-1)^{(|a|+|\xi|)|n|}n\xi(c_i)a\otimes x_i\\
&=(-1)^{(|a|+|\xi|)|n|}(na)\otimes \sum_i\xi(c_i)x_i\\
&=(-1)^{(|a|+|\xi|)|n|}(na)\otimes(\xi\ocap x)\\
&=(a\otimes\xi)\cdot(n\otimes x)
 \end{align*}
As a consequence, we see that the identity map of ${N\otimes X}$ is $\sigma^{A\avop C^*\avop}$-equivariant.

The $\sigma^{AC^*}$-equivariance of $\id^{Y\otimes M}$ is verified by a similar calculation.
  \end{proof}

    \section{Duals of twisting maps}
      \label{Dual twisting maps}

In this section we track the behavior of cup products and cap products under vector space duality.
Under appropriate finiteness conditions on the algebra side, satisfyingly simple expressions are 
obtained for duals of twisted tensor products.  

   \begin{notation}
    \label{ch:notation6}
The notation in \ref{ch:notation1} is in force and $C^*$ is the DG algebra from~\ref{ch:dualActions}.
  \end{notation}
  
  \begin{bfchunk}{Duality.}
  \label{ch:dualActions2}
Let $A$ be a degreewise finite DG algebra.  When $A_{\ll0}=0$ or $A_{\gg0}=0$ the map 
$\varpi^{AA}$ from \eqref{eq:cx2} is bijective and the composed morphisms
  \[
\psi^{A^*}\col A^*\xra{\,(\varphi^A)^*\,}(A\otimes A)^*\xra{\,(\varpi^{AA})^{-1}\,}A^*\otimes A^*
  \quad\text{and}\quad 
\varepsilon^{A^*}=(\eta^A)^*
  \]
turn $A^*$ into a DG coalgebra.  If $M$ and $N$ are degreewise finite and adequate, then 
the maps $\varpi^{AM}$ and $\varpi^{NA}$ from \eqref{eq:cx2} are bijective and the morphisms
  \begin{alignat*}{2}
\psi^{A^*M^*}&:=(\varpi^{AM})^{-1}(\varphi^{AM})^*\col &M^*&\to A^*\otimes M^*
  \\
\psi^{N^*A^*}&:=(\varpi^{NA})^{-1}(\varphi^{NA})^*\col &N^*&\to N^*\otimes A^*
  \end{alignat*}
turn $M^*$ into a left DG $A^*$-comodule and $N^*$ into a right DG $A^*$-comodule.
   \end{bfchunk}

  \begin{bfchunk}{Actions on duals.}
    \label{ch:inducedXiActions}
The following left actions are always defined:
  \begin{enumerate}[\quad\rm(1)]
  \item[{}]
$\Xi_{CA}$ acts on $\Hom(X,M)$ by \ref{ch:cupXM} and on $(N\otimes X)^*$ by \ref{ch:capNX} and
\eqref{eq:coinduced}.
  \item[{}]
$\Xi_{CA}^{\mathsf o}$ acts on $\Hom(Y,N)$ by \ref{ch:cupYN} and on $(Y\otimes M)^*$ 
by \ref{ch:capYM} and \eqref{eq:coinduced}.
   \end{enumerate}
In case $A$ is degreewise finite and satisfies $A_{\ll0}=0$ or $A_{\gg0}=0$, then the DG coalgebra $A^*$ 
of \ref{ch:dualActions2} and the DG algebra $C^*$ of \ref{ch:dualActions} also define DG algebras 
$\Xi_{A^*C^*}$ and $\Xi_{A^*C^*}^{\mathsf o}$; see \ref{ch:Xi} and \ref{ch:Xiop}, respectively.  If, in addition, 
$M$ and $N$ are degreewise finite and adequate for $A$, then in view of \ref{ch:dualActions2} additional left actions are defined:
  \begin{enumerate}[\quad\rm(1)]
     \item[{}]
$\Xi_{A^*C^*}$ acts on $\Hom(M^*,X^*)$ due to \ref{ch:cupXM} and on $N^*\otimes X^*$ due to \ref{ch:capYM}.
    \item[{}]
$\Xi_{A^*C^*}^{\mathsf o}$ acts on $\Hom(N^*,Y^*)$ due to \ref{ch:cupYN} and on $Y^*\otimes M^*$ due to \ref{ch:capNX}.
   \end{enumerate}
   \end{bfchunk}

The DG module structures reviewed above involve actions from four different DG algebras.  
The next result describes various interactions.

  \begin{proposition}
    \label{prop:duality}
If $A$ is degreewise finite and satisfies $A_{\ll0}=0$ or $A_{\gg0}=0$, and both $M$ and $N$ are degreewise finite 
and adequate for $A$, then the maps $\delta$ from \eqref{eq:cx3} and $\varpi$ from \eqref{eq:cx2} 
agree with the actions in \emph{\ref{ch:inducedXiActions}} in the following sense:
  \begin{enumerate}[\rm(1)]
  \item
The map $\delta^{CA}\col\Xi_{CA}\to\Xi_{A^*C^*}$ is an injective morphism of DG algebras.
  \item
There are injective $\delta^{CA}$-equivariant morphisms of left DG modules
   \begin{align*}
 \delta^{XM}&\col\Hom(X,M)\to\Hom(M^*,X^*)
   \\
\varpi^{NX}&\col N^*\otimes X^*\to(N\otimes X)^*
  \end{align*}
  \item
There are injective $(\delta^{CA})^{\mathsf o}$-equivariant morphisms of left DG modules
  \begin{align*}
 \delta^{YN}&\col\Hom(Y,N)\to\Hom(N^*,Y^*)
   \\
\varpi^{YM}&\col Y^*\otimes M^*\to(Y\otimes M)^*
  \end{align*}
  \end{enumerate}
        \end{proposition}
    
  \begin{proof}
All the maps in the proposition are injective by \ref{ch:cxDual}(2) and \ref{ch:cxDual}(3).
 
(1) From $(\eta^A\varepsilon^C)^*=(\varepsilon^C)^*(\eta^A)^*=\eta^{C^*}\varepsilon^{A^*}$ we see that
$\delta^{AC}$ commutes with unit maps.  It commutes with products by (2) applied with $X=C$ and $M=A$.

(2) For $\xi$ in $\Xi_{CA}$ and $\zeta$ in $\Hom(X,M)$, using the formulas in \ref{ch:dualActions} we get
  \begin{align*}
\ub{\delta^{XM}(\xi\scup\zeta)}
&=\ub{(\varphi^{AM}(\xi|\zeta)\psi^{CX})^*}\\
&=(\psi^{CX})^*\ub{(\xi|\zeta)^*}(\varphi^{AM})^*\\
&=\ub{(\psi^{CX})^*\varpi^{CX}}(\xi^*|\zeta^*)\ub{(\varpi^{AM})^{-1}(\varphi^{AM})^*}\\
&=\ub{\varphi^{C^*X^*}(\xi^*|\zeta^*)\psi^{A^*M^*}}\\
&=\delta^{CA}(\xi)\scup\delta^{XM}(\zeta)
    \end{align*}

Using, in addition, the equalities \eqref{eq:coinduced} we obtain
 \begin{align*}
\ub{\lambda^{\Xi_{CA}(N\otimes X)^*}\!(\xi)}\varpi^{NX}
&=\ub{{((\varphi^{NA}|X)(N|\xi| X)(N|\psi^{CX}))^*}}\varpi^{NX}\\
&=(N|\psi^{CX})^*(\ub{N|\xi| X^*})\ub{{(\varphi^{NA}|X)^*}\varpi^{NX}}\\
&=(N|\psi^{CX})^*\ub{{((N|\xi)| X)^*\varpi^{(N\otimes A)X}}}(\ub{(\varphi^{NA})^*}|X^*)\\
&=(N|\psi^{CX})^*\varpi^{(N\otimes C)X}\ub{(N|\xi)^*|X^*)(\varpi^{NA}|X^*)}(\psi^{N^*A^*}|X^*)\\
&=\ub{(N|\psi^{CX})^*}\ub{\varpi^{(N\otimes C)X}(\varpi^{NC}|X^*)}(N^*|\xi^*| X^*)(\psi^{N^*A^*}|X^*)\\
&=\ub{(N^*|(\psi^{CX})^*)\varpi^{N(C\otimes X)}}(N^*|\varpi^{CX})(N^*|\xi^*| X^*)(\psi^{N^*A^*}|X^*)\\
&=\varpi^{NX}\ub{(N^*|(\psi^{CX})^*)(N^*|\varpi^{CX})}(N^*|\xi^*| X^*)(\psi^{N^*A^*}|X^*)\\
&=\varpi^{NX}\ub{(N^*|\varphi^{C^*X^*})(N^*|\xi^*| X^*)(\psi^{N^*A^*}|X^*)}\\
&=\varpi^{NX}\lambda^{\Xi_{A^*C^*}(N^*\otimes X^*)}(\delta^{CA}(\xi))
   \end{align*}

(3) This is established by computations parallel to those presented above.
     \end{proof}

All this leads to the main result of this section.

 \begin{theorem}
    \label{cor:inducedHomTwists} 
Assume that $A$ is degreewise finite and let $A^*$ be the DG coalgebra from \emph{\ref{ch:dualActions2}}.
Let $\tau\col C\to A$ be a $k$-linear map and let $\tau^*\col A^*\to C^*$ be its dual.
  \begin{enumerate}[\rm(1)]
    \item
The map $\tau^*$ is twisting if and only if $\tau$ is twisting.
 \end{enumerate}

When $\tau$ is twisting and both $M$ and $N$ are degreewise finite the following hold.
  \begin{enumerate}[\rm(1)]
    \item[\rm(2)]
The map $\varpi^{NX}$ from \eqref{eq:cx2} induces an injective natural morphism of complexes
 \begin{equation*}
\varpi^{NX}_{\tau}\col 
N^*\avlotimes{\tau^*}X^*\to(N\rotimes{\tau}X)^*
 \end{equation*}
It is bijective if $N$ or $X$ is finite, or $N_{\ll0}=0=X_{\ll0}$, or $N_{\gg0}=0=X_{\gg0}$.
    \item[\rm(3)]
The map $\varpi^{YM}$ from \eqref{eq:cx2} yields an injective natural morphism of complexes
 \begin{equation*}
\varpi^{YM}_{\tau}\col 
Y^*\rotimes{\tau^*}M^*\to(Y\avlotimes{\tau}M)^*
 \end{equation*}
It is bijective if $M$ or $Y$ is finite, or $M_{\ll0}=0=Y_{\ll0}$, or $M_{\gg0}=0=Y_{\gg0}$.   
    \item[\rm(4)]
When $A$ or $C$ is finite, or $A_{<0}=0=\ov C_{\les0}$, or $A_{>0}=0=\ov C_{\ges0}$ the map
$\tau^*$ is acyclic if and only if $\tau$ is acyclic.
  \end{enumerate}
 \end{theorem}
 
  \begin{proof}
(1) Proposition~\ref{prop:duality}(1) and \ref{ch:cxDual}(2) give an isomorphism 
$\Xi_{CA}\to\Xi_{A^*C^*}$ of DG algebras sending $\tau$ to $\tau^*$.  By 
\ref{ch:twisters}(5), the map $\tau^*$ is twisting if and only if $\tau$ is.

(2) The map $\varpi^{NX}\col N^*\otimes X^*\to(N\otimes X)^*$ is an injective natural 
$\delta^{CA}$-equivariant morphisms of left DG modules, see Proposition {\ref{prop:duality}(2)},  
so by \ref{ch:twisters}(5) it induces a morphism of complexes $\varpi^{NX}_\tau$.
The claim about its bijectivity is settled in \ref{ch:cxDual}(3).

(3) This holds for reasons similar to those used to prove (2). 

(4)  The following pairs of conditions are equivalent: 
  \begin{enumerate}[\rm\qquad(1)]
    \item[{\ }]
$\HH(C^*\rotimes{\tau^*}A^*)\cong k$ and $\HH((C\avlotimes{\tau}A)^*)\cong k$, by (2).
    \item[{\ }]
$\HH((C\avlotimes{\tau}A)^*)\cong k$ and $\HH(C\avlotimes{\tau}A)\cong k$, as $?^*$ is a faithfully exact functor.
    \item[{\ }]
$\HH(C\avlotimes{\tau}A)\cong k$ and $\HH(A\rotimes{\tau}C)\cong k$, by Theorem \ref{thm:quismAcy}.  
  \end{enumerate}
Thus, $\HH(C^*\rotimes{\tau^*}A^*)\cong k$ is equivalent to $\HH(A\rotimes{\tau}C)\cong k$, as
desired.
  \end{proof}

\section{Moore duality}
  \label{MooreDuality}

We prove a preliminary version of the theorem announced in the introduction.  It is of 
interest in its own right and is derived from a special case of the equivalence of categories 
of DG modules and DG comodules that goes back to \cite{Mo}, \cite{GM}, \cite{HMS}.

\begin{notation}
  \label{ch:notation7}
Here $B$ is a DG algebra, $\DGM{}{}{B}$ denotes the abelian category of left 
DG $B$-modules, and $K$ and $L$ are used for names of DG $B$-modules.
  \end{notation}

We sketch a construction of the (unbounded) derived category of $B$.

   \begin{bfchunk}{Semifree resolutions.}
    \label{ch:resolutions}
A \emph{semifree resolution} of a left DG $B$-module $L$ is a quasi-isomorphism $F\xra{\simeq}L$ with 
$F$ semifree over $B$.  Such a map always exists, see \cite[\S1]{AH} or \cite[\S6]{FHT2}.  For every 
$L\in\DGM{}{}{B}$ we choose one and denote it by $\res BL\xra{\simeq} L$; when no ambiguity arises we drop the 
reference to $B$ and write simply $\res{}L$.

The properties of semifree DG modules described in \ref{ch:DGMhp} imply that for every morphism 
$L\to K$ there is a unique up to homotopy morphism $\res{}L\to \res{}{K}$, such that the composed 
maps $\res{}L\to \res{}{K}\xra{\simeq} K$ and $\res{}L\xra{\simeq}L\to K$ are homotopic.
  \end{bfchunk}

  \begin{bfchunk}{Derived category of DG modules.}
    \label{ch:dcat}
The objects of the derived category $\dcatdf{}{}{B}$ are the left DG $B$-modules, and 
$\Hom_{\dcatdf{}{}{B}}(L,K)=\HH_0\Hom_B(\res{}L,\res{}K)$; in words: morphisms from $K$ to $L$ are
homotopy classes of morphisms from the chosen semifree resolution of $K$ to that of $L$.  
We write $\simeq$ for isomorphisms in $\dcatdf{}{}{B}$.  

The remarks in \ref{ch:resolutions} imply that $\dcatdf{}{}{B}$ is indeed a category; it is triangulated, with shift 
operator $\shift$ induced from the one in $\DGM{}{}{B}$, see \ref{ch:DGAmod}. The assignment 
$L\mapsto\res BL$ defines a functor $\DGM{}{}{B}\to\dcatdf{}{}{B}$.  It turns quasi-isomorphisms
into isomorphisms and is universal for this property; see \cite{Ke} for details of the construction.
  \end{bfchunk}

  \begin{bfchunk}{Subcategories.}
    \label{ch:subcat}
Triangulated subcategories of $\dcatdf{}{}{B}$ are identified by ornaments on the letter $\mathsf D$.  
Placement matters---superscripts indicate finiteness conditions on underlying $k$-vector 
spaces.  We single out the following full subcategories:

The thick subcategory $\dcatdf{}{perf}{B}$ generated by the DG $B$-module $B$.

The subcategory $\dcatdf{hf}{}{B}$ of all DG modules $L$ with $\rank_k\HH_i(L)$ finite for $i\in\BZ$.

The subcategory $\dcatdf{}{hb}{B}$ of all DG modules $L$ with $\HH_i(L)=0$ for $|i|\gg0$.

\noindent
In case $\HH(B)_{\ll0}=0$, respectively, $\HH(B)_{\gg0}=0$ we also consider:

The subcategory $\dcatdf{}{ha}{B}$ of those $L$ with 
$\HH(L)_{\ll0}=0$, respectively, $\HH(L)_{\gg0}=0$.  

We set $\dcatdf{hf}{ha}{B}=\dcatdf{hf}{}{B}\cap\dcatdf{}{ha}{B}$ and 
$\dcatdf{hf}{hb}{B}=\dcatdf{hf}{}{B}\cap\dcatdf{}{hb}{B}$.

\noindent
In case $B_{\ll0}=0$, respectively, $B_{\gg0}=0$ we also consider:

The subcategory $\dcatdf{}{a}{B}$ of those $L$ with 
$L_{\ll0}=0$, respectively, $L_{\gg0}=0$.  

We set $\dcatdf{f}{a}{B}=\dcatdf{f}{}{B}\cap\dcatdf{}{a}{B}$ and 
$\dcatdf{f}{b}{B}=\dcatdf{f}{}{B}\cap\dcatdf{}{b}{B}$.
  \end{bfchunk}

Recall that an additive functor between triangulated categories is said to be \emph{exact} if it commutes 
with exact triangles and with shifts.  

    \begin{theorem}
    \label{thm:bggC}
Let $A$ be a degreewise finite augmented DG algebra and $C$ a degreewise finite
coaugmented DG coalgebra satisfying the condition
  \[
\text{\rm(p)}\quad
\ov{A}_{\les0}=0=\ov C_{\les1} 
\qquad\text{respectively}\qquad
\text{\rm(n)}\quad 
\ov{A}_{\ges-1}=0=\ov C_{\ges0}
  \]
Let $\tau\col C\to A$ (and hence also $\tau^*$ by Theorem \emph{\ref{cor:inducedHomTwists}(4)}) be an acyclic twisting map.

The functors $(A\rotimes{\tau}\mathrm{?}^*)$ and $(C^*\rotimes{\tau^*}\mathrm?)$ localize to an adjoint exact equivalence 
  \begin{equation}
    \label{eq:bggC1}
\xymatrixcolsep{3pc}
\xymatrixrowsep{2pc} 
\xymatrix{
C^*\rotimes{\tau^*}\mathrm{?}^*\col
\dcatdf{f}{a}{A}\opp
\ar@{->}[r]<1.1ex>
&\dcatdf{f}{a}{C^*}
\ar@{->}[l]<1ex>_-{\equiv}
}
:\!A\rotimes{\tau}\mathrm{?}^*
  \end{equation}
The latter restricts to exact equivalences
   \begin{align}
    \label{eq:bggC2}
\xymatrixcolsep{3pc}
\xymatrixrowsep{2pc} 
\xymatrix{
\dcatdf{}{perf}{A}\opp
\ar@{->}[r]<1.1ex>
&\dcatdf{f}{b}{C^*}
\ar@{->}[l]<1ex>_-{\equiv}
}
\quad\text{and}\quad 
\xymatrix{
\dcatdf{f}{b}{A}\opp
\ar@{->}[r]<1.1ex>
&\dcatdf{}{perf}{C^*}
\ar@{->}[l]<1ex>_-{\equiv}
}
  \end{align}
and take the following values:
  \begin{equation}
    \label{eq:bggC3}
A\rotimes{\tau}k^* \simeq A
\quad
A\rotimes{\tau}C^{**}\simeq k
   \quad \text{and}\quad
C^*\rotimes{\tau^*}k^*\simeq C^*
\quad
C^*\rotimes{\tau^*}A^*\simeq k
  \end{equation}
    \end{theorem}

In accordance with the notation in \ref{ch:subcat}, we write $\DGM{f}{a}{B}$ for the abelian category 
of left DG $B$-modules whose objects are the degreewise finite adequate DG modules.

   \begin{lemma}
    \label{lem:factorizationM}
For $A$ and $C$ as in Theorem \emph{\ref{thm:bggC}} there is an adjoint equivalence
  \begin{equation}
    \label{eq:factorizationM}
\xymatrixcolsep{3pc}
\xymatrixrowsep{2pc} 
\xymatrix{
C^*\rotimes{\tau^*}\mathrm{?}^*\col
\DGM{f}{a}{A}\opp
\ar@{->}[r]<1.1ex>_-{\equiv}
&\DGM{f}{a}{C^*}
\ar@{->}[l]<1ex>
}
:\!A\rotimes{\tau}\mathrm{?}^*
  \end{equation}
    \end{lemma}

   \begin{proof}
For $M$ in $\DGM{f}{a}{A}$ and $L$ in $\DGM{f}{a}{C^*}$ there is a natural isomorphism of complexes 
$\Hom(C^*\rotimes{\tau^*}M^*,L)\cong\Hom(L^*,(C^*\rotimes{\tau^*}M^*)^*)$;  see \ref{ch:cxDual}(2).
Due to the definition of the action of $C^*$ in \ref{ch:dualActions}, it yields the first isomorphism in the string
  \begin{align*}
\Hom_{C^*}(C^*\rotimes{\tau^*}M^*,L)
&\cong\Hom_{C}(L^*,(C^*\rotimes{\tau^*}M^*)^*)\\
&\cong\Hom_{C}(L^*,C\avlotimes{\tau}M)\\
&\cong\Hom_{A}(A\rotimes{\tau}L^*,M)
  \end{align*}
The other two isomorphisms come from Theorem \ref{cor:inducedHomTwists}(3) and \ref{ch:twistAdj}.
Morphisms of DG modules are the degree zero cycles in the complex of homomorphisms, so we get
  \begin{align*}
\Hom_{\DGM{f}{a}{C^*}}(C^*\rotimes{\tau^*}M^*,L)
&\cong\Hom_{\DGM{f}{a}{A}}(A\rotimes{\tau}L^*,M)
  \end{align*}
The space on the right is $\Hom_{\DGM{f}{a}{A}\opp}(M,A\rotimes{\tau}L^*)$, whence the desired assertion.
  \end{proof}

  \begin{proof}[Proof of Theorem \emph{\ref{thm:bggC}}]
Let $M$ be a DG module in $\DGM{f}{a}{A}$.  If $\HH(M)=0$, then
  \[
\HH(C^*\rotimes{\tau^*}M^*)\cong\HH((C\avlotimes{\tau}M)^*)\cong(\HH(C\avlotimes{\tau}M))^*\cong(\HH((C\avlotimes{\tau}A)\otimes_AM))^*=0
  \]
where Theorem \ref{cor:inducedHomTwists}(3) gives the first isomorphism and the other are standard; the equality holds as 
$C\avlotimes{\tau}A$ is semifree by Lemma \ref{lem:semifree}, so $(C\avlotimes{\tau}A)\otimes_AM\simeq0$ by~\ref{ch:DGMhp}.  

It follows that $C^*\rotimes{\tau^*}?^*$  defines an exact functor 
$C^*\rotimes{\tau^*}\mathrm{?}^*\col\dcatdf{f}{a}{A}\opp\to\dcatdf{f}{a}{C^*}$.
An exact functor $(A\rotimes{\tau}\mathrm{?}^*)$ in the opposite direction is similarly obtained.
These functors form an adjoint pair because they are induced by such a pair.  

To prove that they are quasi-inverse we show that the adjunction unit and counit are isomorphisms.
The unit yields in $\dcatdf{f}{a}{A}\opp$ morphisms $M\to A\ltensor_{\tau}(C^*\rotimes{\tau^*}M^*)^*$
induced by the composed morphism of left DG $A$-modules
  \begin{gather*}
A\rotimes{\tau}(C^*\rotimes{\tau^*}M^*)^*
\xra{\,A\otimes((\varpi^{CM}_{\tau})^*)^{-1}\,} A\rotimes{\tau}C\avlotimes{\tau}M
\xra{\,\varepsilon^{ACM}\,}M
  \end{gather*}
with $\varpi^{CM}_{\tau}$ from Theorem \ref{cor:inducedHomTwists}(3) and $\varepsilon^{ACM}$ 
from~\eqref{eq:quismAugACM}.  As $\tau$ is acyclic $\varepsilon^{ACM}$ is a quasi-isomorphism 
by Theorem \ref{thm:quismAcy}, so the unit is an isomorphism.  

The adjunction counit yields a morphism $C^*\rotimes{\tau^*}(A\ltensor_{\tau}L^*)^*\to L$  in 
$\dcatdf{f}{a}{C^*}$, which is induced by the composed morphism of left DG $C^*$-modules
  \begin{gather*}
C^*\rotimes{\tau^*}(A\rotimes{\tau}L^*)^*
\xra{\,C^*\otimes((\varpi^{AL^*}_{\tau})^*)^{-1}\,} C^*\rotimes{\tau^*}A^*\avlotimes{\tau^*}L
\xra{\,\varepsilon^{C^*A^*L}\,}L
  \end{gather*}
Here $\varpi^{AL^*}_{\tau}$ is given by Theorem \ref{cor:inducedHomTwists}(2) applied with $X=L{}^*$,
and $\varepsilon^{C^*A^*L}$ by \eqref{eq:quismAugACM} applied with $C^*$ in place of $A$ and $A^*$ 
in place of $C$. Since $\tau^*$ is acyclic, $\varepsilon^{C^*A^*L}$ is a quasi-isomorphism by Theorem 
\ref{thm:quismAcy}, and thus the counit is an isomorphism.  

Note next that $k\simeq A\rotimes{\tau}C$ holds in $\dcatdf{f}{a}{A}\opp$ because $\tau$ is acyclic.  
We get 
  \[
C^*\rotimes{\tau^*}k\simeq C^*\rotimes{\tau^*}(A\rotimes{\tau}C)^*\simeq 
C^*\rotimes{\tau^*}A^*\avlotimes{\tau^*}C^*\simeq C^*
  \]
by invoking Theorems \ref{cor:inducedHomTwists}(2) and \ref{thm:quismAcy} for the last two isomorphisms.  
On the other hand, since $A$ is semifree we have $C^*\rotimes{\tau^*}A\simeq C^*\rotimes{\tau^*}A^*\simeq k$,
using Theorem \ref{thm:quismAcy} once again.  Similar computations give $A\rotimes{\tau} k\simeq A$ and 
$A\rotimes{\tau}C^*\simeq k$.

These isomorphisms imply that $C^*\rotimes{\tau^*}\mathrm{?}$ and $A\rotimes{\tau}\mathrm{?}$ 
restrict to an equivalence of the thick subcategory of $\dcatdf{f}{a}{C^*}$ generated by $k$ with the 
thick subcategory of $\dcatdf{f}{f}{A}$ generated by $A$.  The latter is $\dcatdf{}{perf}{A}$ by definition, 
while the former equals $\dcatdf{f}{a}{C^*}$ by \cite[6.4]{ABIM}.  The other equivalence in 
\eqref{eq:bggC2} holds by symmetry.
  \end{proof}

    \section{Composition products}
      \label{CompositionProducts}

Composition products of derived Hom functors are chain level maps that induce Yoneda 
products in cohomology.  The material in this section is basic and known, but it is spread
out over sources adopting different sets of hypotheses.  

The discussion is geared towards applications to two distinct goals.  The first one is 
to translate Moore duality into Koszul duality for adequate degreewise finite DG modules 
over degreewise finite DG algebras.  The second is to replace both finiteness conditions 
on the objects by the respective conditions on their homology.
 
   \begin{notation}
    \label{ch:notation15}
In this section $A$ and $B$ denote DG algebras and ${}_Ak$ is a fixed left DG $A$-module.
The general assumptions and notation from \ref{ch:notation1} are in force. 

The triangulated categories used below are described in \ref{ch:subcat}.
  \end{notation}

 \begin{bfchunk}{Morphisms.}
   \label{ch:morph}
A morphism  $\alpha\col A\to B$ of DG algebras yields an adjoint pair 
\begin{equation}
    \label{eq:morph1}
\xymatrixcolsep{2pc}
\xymatrixrowsep{2pc} 
\mathsf{L}^{\alpha}\col
\xymatrix{
\dcatdf{}{}{A}
\ar@{->}[r]<.5ex>
&\dcatdf{}{}{B}
\ar@{->}[l]<.5ex>
}
:\!\mathsf{R}^{\alpha}
  \end{equation}
of exact functors described as follows: $\mathsf{R}^{\alpha}$ is the forgetful functor that replaces any action 
of $B$ by the action of $A$ obtained through $\alpha$, and $\mathsf{L}^{\alpha}(\mathrm?)$ is the derived 
tensor product $B\ltensor_A\mathrm{?}$, defined on objects by the assignment 
$L\mapsto L'\otimes_B\res BL$.

When $\alpha$ is a quasi-isomorphism these functors are quasi-inverse equivalences.

Indeed, recall from \ref{ch:resolutions} that a semifree resolution $\epsilon^M\col\res{}M\xra{\simeq}M$ 
has been chosen for every DG $A$-module $M$ and that $\mathsf{L}^{\alpha}(M)=B\otimes_A\res{}M$.
Thus, the unit of the adjunction is the homotopy class of the map $\alpha\otimes\res{}M$, which is
a quasi-isomorphism because $\alpha$ is one and $\res{}M$ is semifree.  The counit
is induced by the map $B\otimes_A\res{}{M}\to M$ given by $(b\otimes m)\mapsto b\epsilon^M(m)$;
it is a quasi-isomorphism because its precomposition with $\alpha\otimes_A\res{}{M}$ 
equals $A\otimes_A\res{}{M}=\res{}{M}\xra\simeq M$.

These quasi-isomorphisms also imply that the functors restrict to equivalences
   \begin{equation}
    \label{eq:morph2}
\dcatdf{hf}{ha}{A}\equiv\dcatdf{hf}{ha}{B}\qquad
\dcatdf{hf}{hb}{A}\equiv\dcatdf{hf}{hb}{B}\qquad
\dcatdf{}{perf}{A}\equiv\dcatdf{}{perf}{B}
   \end{equation}
 \end{bfchunk}
 
   \begin{bfchunk}{Formal DG algebras.}
    \label{ch:formal}
The DG algebra $A$ is said to be \emph{formal} if there exists a string of quasi-isomorphisms of DG 
algebras linking $A$ and $\HH(A)$.  When $A$ is formal the equivalences \eqref{eq:morph2} hold 
with $\HH(A)$ in place of $B$.
  \end{bfchunk}

  \begin{bfchunk}{Derived homomorphism functors.}
    \label{ch:derivedComposition}
An exact functor
  \begin{align*}
\rhom_A(\text{\rm>}\,,{\mathrm?})&\col\dcatdf{}{}{A}\opp\times\dcatdf{}{}{A}\to\dcatdf{}{}{k}
  \end{align*}
is defined by $(M,{}_Ak)\mapsto\Hom_A(\res{}M,\res{}{{}_Ak})$, and one to graded $k$-spaces by
  \[
\Ext_A(M,{}_Ak)=\HH(\rhom_A(M,{}_Ak))
  \]

Set $E=\End_A(\res {}{{}_Ak})$.  Composition of homomorphisms turns $E$ into a DG algebra and gives 
$\rhom_A(M,{{}_Ak})$ a structure of left DG $E$-module that is natural in $M$:  Each morphism of DG 
$A$-modules $M\to M'$ defines a unique up to homotopy morphism $\res{}M\to\res{}{M'}$ of DG $A$-modules, 
so the induced morphism of complexes $ \Hom_A(\res{}{M'},\res{}{{}_Ak})\to \Hom_A(\res{}M,\res{}{{}_Ak})$
is unique up to homotopy.  It is also $E$-linear, hence the assignment $M\mapsto\rhom_A(M,{{}_Ak})$ 
yields an exact functor
  \[
\rhom_A(\mathrm{?},{{}_Ak})\col\dcatdf{}{}A\opp\to\dcatdf{}{}{\rhom_A({}_Ak,{}_Ak)}
  \]
      \end{bfchunk}

Next we show that other choices of resolution of ${}_Ak$ produce comparable results.  

   \begin{lemma} 
      \label{ch:CompositionM}
Let $F'\simeq {}_Ak$ be a semifree resolution and set $E'=\End_A(F')$.

There exists a sequence of quasi-isomorphisms of DG algebras linking $E$ and $E'$.  
It induces an exact equivalence $\dcatdf{}{}{E}\equiv\dcatdf{}{}{E'}$
that restricts to equivalences
   \begin{align}
    \label{eq:CompositionM}
\dcatdf{hf}{ha}{E}\equiv\dcatdf{hf}{ha}{E'}\qquad
\dcatdf{hf}{hb}{E}\equiv\dcatdf{hf}{ha}{E'}\qquad
 \dcatdf{}{perf}{E}\equiv\dcatdf{}{perf}{E'}
 \end{align}
    \end{lemma}

   \begin{proof}
In view of \ref{ch:morph}, it suffices to prove the first assertion.  Set $F=\res{}{{}_Ak}$.  

Since both $F$ and $F'$ are semifree resolutions of ${}_Ak$, there is a homotopy equivalence
$\phi\col F\to F'$.  Define a DG $A$-module $P$ by $P\nat=(F\oplus \shift^{-1}F'\oplus F')\nat$ and 
  \[
\dd^P(f,f'',f')=(\dd^{F}(f),\phi(f)-\dd^{F'}(f'')+f',\dd^{F'}(f'))
  \]
Note that $P$ is semifree and the canonical maps $P\to F$ and $P\to F'$ are surjective morphisms of DG 
modules.  Their kernels are isomorphic to the mapping cones of $\shift^{-1}(\id^{F'})$ and $\shift^{-1}(\phi)$, 
respectively, so both maps are quasi-isomorphisms.  

By symmetry, it suffices to deal with $P\to F$.  Being a surjective quasi-isomorphism onto a semifree DG module
it has a right inverse, so we may assume $P=F\oplus G$ with a DG module $G$ quasi-isomorphic to $0$.  
The composed morphism $\epsilon\col P\to F\to P$ is an idempotent in $\End_A(P)$, so 
$\alpha\mapsto\epsilon\alpha\epsilon$ is a surjective morphism of DG algebras with image isomorphic 
to $\End_A(F)$. It is a quasi-isomorphism, as its kernel is a direct sum of the acyclic 
complexes $\Hom_A(F,G)$, $\Hom_A(G,G)$, and $\Hom_A(G,F)$.
  \end{proof}

Next we look at the functoriality of derived composition products.

 \begin{bfchunk}{Naturality.}
    \label{ch:forgetFun2}
Let $\alpha\col A\to B$ be a morphism of DG algebras and set
  \begin{equation}
    \label{eq:forgetFunSet}
\mathsf{E}(A)=\Hom_{A}({}_Ak,{}_Ak)
\quad\text{and}\quad
\mathsf{E}(B)=\Hom_{B}({}_Bk,{}_Bk)
  \end{equation}
where ${}_Bk\coloneqq B\otimes_A\res{}{{}_Ak}$.
The map $\phi\mapsto B\otimes_A\phi$ is a morphism of DG algebras
  \begin{equation}
    \label{eq:forgetFun2}
\mathsf{E}(\alpha)\col \mathsf{E}(A)\to \mathsf{E}(B)
  \end{equation}
and an $\mathsf{E}(\alpha)$-equivariant morphism of left DG modules
  \begin{equation}
    \label{eq:forgetFun3}
\alpha^M\col\Hom_{A}(\res{}M,{}_Ak)\to\Hom_{B}(B\otimes_{A}\res{}M,{}_Bk)
  \end{equation}
As $B\otimes_{A}\res{}M$ is semifree over $B$, the last map defines a natural transformation
  \begin{equation}
    \label{eq:forgetFun4}
\rhom_{A}(\mathrm{?},{{}_Ak})\to\mathsf{R}^{\mathsf{E}(\alpha)}\rhom_{B}(\mathsf{L}^{\alpha}(\mathrm{?}),{}_Bk)
  \end{equation}
of exact functors $\dcatdf{}{}{A}\to\dcatdf{}{}{\mathsf{E}(B)}$.
  \end{bfchunk}

  \begin{lemma} 
    \label{lem:exactFun}
If $\alpha\col A\to B$ is a quasi-isomorphism of DG algebras, then so is the map
$\mathsf{E}(\alpha)$ in \eqref{eq:forgetFun2}. It induces adjoint quasi-inverse equivalences
\begin{equation}
    \label{eq:exactFun1}
\xymatrixcolsep{3pc}
\xymatrixrowsep{2pc} 
\mathsf{L}^{\mathsf{E}(\alpha)}\col
\xymatrix{
\dcatdf{}{}{\mathsf{E}(A)}
\ar@{->}[r]<1.1ex>_-{\equiv}
&\dcatdf{}{}{\mathsf{E}(B)}
\ar@{->}[l]<1ex>
}
:\!\mathsf{R}^{\mathsf{E}(\alpha)}
   \end{equation}
that restrict to equivalences
\begin{equation}
    \label{eq:exactFun2}
  \begin{gathered}
\dcatdf{hf}{ha}{\mathsf{E}(A)}\equiv\dcatdf{hf}{ha}{\mathsf{E}(B)}\\
\dcatdf{hf}{hb}{\mathsf{E}(A)}\equiv\dcatdf{hf}{hb}{\mathsf{E}(B)}\qquad
\dcatdf{}{perf}{\mathsf{E}(A)}\equiv\dcatdf{}{perf}{\mathsf{E}(B)}
  \end{gathered}
 \end{equation}
Furthermore, the natural transformation \eqref{eq:forgetFun4} is a natural isomorphism.
  \end{lemma}

  \begin{proof}
The morphism $\alpha^M $ from \eqref{eq:forgetFun3} composed with the isomorphism
 \[
\Hom_{B}(B\otimes_A\res{}M,B\otimes_A\res{}{{}_Ak})\xra{\cong}\Hom_{A}(\res{}M,B\otimes_A\res{}{{}_Ak})
  \]
equals $\Hom_{A}(\res{}M,\alpha\otimes_A\res{}{{}_Ak})$.  The latter is a quasi-isomorphism as $\alpha$ 
is one and both $\res{}{{}_Ak}$ and $\res{}M$ are semifree.  Thus, $\alpha^M$ is a quasi-isomorphism.

In other words, \eqref{eq:forgetFun4} is a natural isomorphism. Applied to $M={}_Ak$ it shows that $\mathsf{E}(\alpha)$ 
is a quasi-isomorphism.  This implies the desired equivalences, see \ref{ch:morph}.
  \end{proof}

We finish with a fact that depends on the properties of the DG algebra in play. 

 \begin{bfchunk}{Controlled resolutions.}
    \label{lem:extensions}
Assume that $A$ is augmented, degreewise finite, and satifies $\ov A_{\les0}=0$ or $\ov A_{\ges-1}=0$.
Every DG module $M$ in $\dcatdf{hf}{ha}A$ then admits a semi-free resolution 
$\mathsf{F^f}({M})\simeq {M}$ with $\mathsf{F^f}({M})\in\dcatdf{hf}{ha}A$.  

Indeed, such a resolution can be obtained by mimicking the inductive construction
of a free resolution of a bounded below complex over a noetherian ring, where an 
existing partial resolution is modified by additions of free modules of finite rank, 
and at most two modifications are needed in any given degree.

Thus, the assignment ${M}\mapsto\mathsf{F}^\mathsf{f}_A({M})$ yields an 
exact functor $\mathsf{F}^\mathsf{f}_A\col\dcatdf{hf}{ha}{A}\to\dcatdf{f}{a}{A}$
that is quasi-inverse to the full embedding $\dcatdf{f}{a}{A}\subset\dcatdf{hf}{ha}{A}$.  
    \end{bfchunk}

\section{Koszul duality}
  \label{KoszulDuality}

In this section we prove the theorem stated in the introduction.

  \begin{theorem}
   \label{thm:bgg}
Let $B$ be a DG algebra such that $\HH(B)$ is degreewise finite, augmented, and satisfies the condition
  \[
\text{\rm(hp)}\quad
\ov{\HH(B)}_{\les0}=0 
\qquad\text{respectively}\qquad
\text{\rm(hn)}\quad 
\ov{\HH(B)}_{\ges-1}=0
  \]

There exists then a DG $B$-module $K$ with $\rank_k\HH(K)=1$, and any DG module with this
property is isomorphic in $\dcatdf{}{}{B}$ to some shift of $K$.

The exact functor $\rhom_B(?,K)$ restricts to an exact equivalence
  \begin{equation}
    \label{eq:bgg1}
\xymatrixcolsep{1.5pc}
\xymatrixrowsep{2pc} 
\xymatrix{
\rhom_B(?,K)\col \dcatdf{hf}{ha}{B}\opp
\ar@{->}[r]^-{\equiv}
&\dcatdf{hf}{ha}{\mathsf{E}(B)}
}
  \end{equation}
where $\mathsf{E}(B)={\rhom_B(K,K)}$, and further restricts to exact equivalences
  \begin{equation}
    \label{eq:bgg2}
\xymatrixcolsep{1.5pc}
\xymatrixrowsep{2pc} 
\xymatrix{
\dcatdf{hf}{hb}{B}\opp
\ar@{->}[r]^-{\equiv}
&\dcatdf{}{perf}{\mathsf{E}(B)}
}
\quad\text{and}\quad 
\xymatrix{
\dcatdf{}{perf}{B}\opp
\ar@{->}[r]^-{\equiv}
&\dcatdf{hf}{hb}{\mathsf{E}(B)}
}
  \end{equation}
   \end{theorem}

A right adjoint of the functor \eqref{eq:bgg1} can be read off the proof; see Remark \ref{rem:right}.

The hypotheses of the theorem and the additional structures introduced in preparation for the proof are 
kept throughout the section.  We start with the existence and uniqueness of $K$.  It is due to Dwyer, 
Greenlees and Iyengar~\cite[3.2, 3.3, 3.9]{DGI}, but in the proof of the theorem we use the specific
form described Lemma \ref{lem:simple} 

   \begin{bfchunk}{Augmented models.}
      \label{ch:models}
There exists a quasi-isomorphism $\alpha\col A\to B$ of DG algebras such that $A$ is augmented, 
degreewise finite, and satisfies
 \begin{equation}
\text{\rm(p)}\quad
\ov{A}_{\les0}=0
\qquad\text{respectively}\qquad
\text{\rm(n)}\quad 
\ov{A}_{\ges-1}=0
  \end{equation}
See Lemaire \cite[4.2.4]{Le}, respectively, F\'elix, Halperin and Thomas~\cite[4.2]{FHT1}.  

Recall from \ref{ch:AugA} that there exists a unique augmentation $\varepsilon^A\col A\to k$.  
  \end{bfchunk}

  \begin{lemma}
    \label{lem:simple}
A DG $B$-module $K$ has $\rank_k\HH(K)=1$ if and only if there exist an integer $j$ and an isomorphism
$K\simeq\shift^j(B\ltensor_Ak)$ in $\dcatdf{}{}B$ .
    \end{lemma}
 
    \begin{proof}
By parts (1) and (2) of Lemma \ref{lem:exactFun} the functor $(B\ltensor_A\mathrm{?})\col \dcatdf{}{}{A}\to\dcatdf{}{}{B}$ 
is an exact equivalence and $\HH(\alpha\ltensor_AM)\col\HH(M)\to\HH(B\ltensor_AM)$ is a $k$-linear isomorphism.  
This shows that $B\ltensor\varepsilon^A$ is a quasi-augmentation, and that to finish the proof it suffices to show that
$\HH(K)\cong k$ as vector spaces implies $K\simeq k$ in $\dcatdf{}{}{A}$.

Define $K'\subseteq K$ as follows: $K'_{i}=K_i$ for $i\ge1$, $K'_{i}=0$ for $i\le-1$, and $K'_0=\im(\dd^K_1)$; 
respectively, $K'_{i}=0$ for $i\ge1$, $K'_{i}=K_i$ for $i\le-1$, and $K'_0$ is a subspace with 
$K'_0\oplus \Ker(\dd^K_0)=K_0$.  The hypotheses on $A$ imply that $K'$ is a DG submodule, and those 
on $K$ that $K\twoheadrightarrow K/K'$ is a quasi-isomorphism.  Let $z\in(K/K')_0$ be a cycle whose 
class generates $\HH_0(K/K')$ and $\kappa\col A\to K/K'$ be the morphism of DG $A$-modules 
with $\kappa(1)=z$.  Since $\kappa(\ov A)=0$ holds for degree reasons, we obtain a morphism 
$\kappa'\col k\to K/K'$ with $\HH(\kappa')\ne0$, so $\kappa'$ is a quasi-isomorphism. 
  \end{proof}

We need yet another avatar of the dual DG algebra of a DG coalgebra.

   \begin{bfchunk}{Homomorphisms of comodules.}
      \label{ch:dualComposition}
Let $C$ be a DG coalgebra,

The equality $C^*=\Xi_{Ck}$ from \ref{ch:dualActions} and the maps $\omega^{AC}$ from 
\eqref{eq:twistAdjXM} with $A=k=M$ yield an isomorphism of DG algebras
  \[
\omega^{kC}\col C^*\cong\End_C({}_CC)
  \quad\text{given by}\quad
\xi\mapsto(C\otimes\xi)\psi^C
  \]
and for every left DG $C$-comodule $X$ an equivariant isomorphism
  \[
\omega^{kC}\col X^*{\,\cong\,}\Hom_C(X,C)
  \quad\text{given by}\quad
\zeta\mapsto(X\otimes\zeta)\psi^{CX}
  \]
of left DG $C^*$-modules, with products on the right-hand sides given by composition.

Indeed, $\omega^{kC}$ is bijective by Lemma~\ref{lem:omega}.  For $\xi\in C^*$ and $\zeta\in X^*$ we get
\begin{align*}
\ub{\omega^{kC}}\ub{(\xi\scup\zeta)}
&=\ub{(C|(\xi\scup\zeta)}\psi^{CX}\\
&=\ub{(C|\xi|\zeta)}\ub{(C|\psi^{CX})\psi^{CX}}\\
&=(C|\xi)\ub{(C|C|\zeta)(\psi^{C}|X)}\psi^{CX}\\
&=\ub{(C|\xi)\psi^{C}}\ub{(C|\zeta)\psi^{CX}}\\
&=\omega^{kC}(\xi)\circ\omega^{kC}(\zeta)
  \end{align*}
    \end{bfchunk}

  \begin{lemma}
     \label{lem:dualCompositionA}
Let $A$ be an augmented DG algebra and $C$ a coaugmented DG  coalgebra such that 
$\ov A:=\Ker(\varepsilon^A)$ and $\ov C:=\Ker(\varepsilon^C)$ satisfy the conditions
  \[
\text{\rm(p)}\quad
\ov A_{\les0}=0=\ov C_{\les1} 
\qquad\text{respectively}\qquad
\text{\rm(n)}\quad
\ov A_{\ges-1}=0=\ov C_{\ges0}
  \]

For $F=A\rotimes{\tau}C$ the maps in \emph{\ref{ch:dualComposition}} define a morphism of DG algebras
  \begin{equation}
    \label{eq:dualCompositionA1}
\varkappa^{C}\col C^*\to \End_A(F)
  \quad\text{by}\quad
\varkappa^{C}(\xi)=A\otimes\omega^{kC}(\xi)
  \end{equation}
and for each left DG $A$-module a natural $\varkappa^{C}$-equivariant morphism of DG modules
  \begin{equation}
    \label{eq:dualCompositionA2}
\varkappa^{CM}\col(C\avlotimes{\tau}M)^*\to\Hom_A(F\avlotimes{\tau}M,F)
  \quad\text{by}\quad
\varkappa^{CM}(\zeta)=A\otimes\omega^{kC}(\zeta)
  \end{equation}
When $\tau$ is acyclic they are quasi-isomorphisms, and so yield an isomorphism $\HH(\varkappa^C)$ 
of graded algebras and an $\HH(\varkappa^C)$-equivariant isomorphism of graded modules:
  \begin{equation}
    \label{eq:dualCompositionA3}
\HH(C^*)\cong\Ext_A(k,k)
  \quad\text{and}\quad
\HH((C\avlotimes{\tau}M)^*)\cong\Ext_A(M,k)
  \end{equation}
     \end{lemma} 

   \begin{proof} 
The multiplicative properties of $\varkappa^{C}$ and $\varkappa^{CM}$ follow from those of $\omega^{kC}$,
see \ref{ch:dualComposition}, so we only need to show that when $\tau$ is acyclic $\varkappa^{CM}$ 
is a quasi-isomorphism.

Set $X=C\avlotimes{\tau}M$.  Formula \eqref{eq:ocap} gives $\lambda^{\Xi(A\otimes X)}(\tau)\subseteq\ov A\otimes X$,
where $\ov A=\Ker(\varepsilon^A)$.  It follows that  $\Hom_A(\lambda^{\Xi\avop(A\otimes X)}(\tau),k)$ is equal to zero.  
In view of \eqref{eq:twistXi1} the equality $F\avlotimes{\tau}M=A\rotimes{\tau}X$ yields the equality in the diagram
of complexes
   \[
\xymatrixcolsep{3pc}
\xymatrixrowsep{2pc} 
\xymatrix{
(C\avlotimes{\tau}M)^*
\ar@{->}[r]^-{\varkappa^{CM}}
\ar@{<-}[d]^{\cong}
&\Hom_A(F\avlotimes{\tau}M,F)
\ar@{->}[d]^-{\Hom_A(F\avlotimes{\tau}M,\varepsilon^{AC})}_-{\simeq}
\\
\Hom_{A}(A\otimes C\avlotimes{\tau}M,k)
\ar@{=}[r]
&\Hom_A(F\avlotimes{\tau}M,k)
}
  \]
The isomorphism is standard and the quasi-isomorphism is due to the semifreeness of $F$.
Since the diagram commutes we conclude that $\varkappa^{CM}$ is a quasi-isomorphism.  

Now \eqref{eq:dualCompositionA3} follows because $F\avlotimes{\tau}M\to M$ is a semifree resolution
by  \ref{ch:naturalRes}.
   \end{proof} 

\stepcounter{theorem}

  \begin{proof}[Proof of Theorem \emph{\ref{thm:bgg}}]
Choose, by \ref{ch:models}, a quasi-isomorphism  $\alpha\col A\to B$ of DG algebras with $A$ 
degreewise finite, augmented and satisfying $\ov{A}_{\les0}=0$, respectively, $\ov{A}_{\ges-1}=0$.
In view of Lemmas \ref{lem:simple} and \ref{ch:CompositionM} we may assume $K= B\ltensor_Ak$.  

Choose a degreewise finite coaugmented DG coalgebra $C$ with $\ov C_{\les1} =0$, respectively,
$\ov C_{\ges0}=0$ and an acyclic twisting map $\tau\col C\to A$; e.~g., $C=\babar A$, see~\ref{ch:Bar}. 

We proceed as follows to assemble the diagram of exact functors of triangulated categories, displayed
in \eqref{eq:diagram}:  Theorem \ref{thm:bggC} provides the equivalence $C^*\rotimes{\tau^*}\mathrm{?}^*$.  
The equivalences $\mathsf{F}^\mathsf{f}_{A}$ and $\subset$ come from \ref{lem:extensions}.  
Since $\mathsf{R}^{\varkappa}$ and $\mathsf{L}^{\alpha}$ are induced by the quasi-isomorphisms of DG algebras
$\varkappa$ (from Lemma \ref{lem:dualCompositionA}) and $\alpha$ (available by construction), they 
are equivalences by \ref{ch:morph}.  Lemma \ref{lem:exactFun} gives the equivalence $\mathsf{R}^{\mathsf{E}(\alpha)}$.

The two functors $\dcatdf{f}{a}{A}\opp\to\dcatdf{hf}{ha}{C^*}$ defined by the upper rectangle are isomorphic
due to the natural quasi-isomorphism \eqref{eq:dualCompositionA2}.  As $F\avlotimes{\tau}M\to M$ is a semi\-free 
resolution by \ref{ch:naturalRes}, the two functors $\dcatdf{hf}{ha}{A}\opp\to\dcatdf{hf}{ha}{\mathsf{E}(A)}$ defined by the 
triangle are  isomorphic by Lemma \ref{ch:CompositionM}. On the other hand, both functors 
$\dcatdf{hf}{ha}{A}\opp\to\dcatdf{hf}{ha}{\mathsf{E}(A)}$ defined by the lower rectangle are isomorphic by Lemma \ref{lem:exactFun}.
\begin{equation}
   \label{eq:diagram}
\xymatrixcolsep{5pc}
\xymatrixrowsep{3.5pc} 
\begin{gathered}
\xymatrix{
&\dcatdf{f}{a}{A}\opp
\ar@{->}[r]^-{C^*\rotimes{\tau^*}\mathrm{?}^*}_-{\equiv}
\ar@{->}[d]^-{\Hom_A(F\avlotimes{\tau}{\mathrm?},F)}
& \dcatdf{f}{a}{C^*}
\ar@{->}[d]^-{\bigcap}_-{\equiv}
  \\
\dcatdf{hf}{ha}{A}\opp
\ar@{->}[ur]^-{(\mathsf{F}^\mathsf{f}_A)\opp}_-{\equiv}
\ar@{->}[r]^-{\rhom_{A}(\mathrm{?},k)}
\ar@{->}[d]_-{\mathsf{L}^{\alpha}}^-{\equiv}
&\dcatdf{hf}{ha}{\mathsf{E}(A)}
\ar@{->}[r]^-{\mathsf{R}^{\varkappa}}_-{\equiv}
\ar@{<-}[d]^-{\mathsf{R}^{\mathsf{E}(\alpha)}}_-{\equiv}
&\dcatdf{hf}{ha}{C^*}
  \\
\dcatdf{hf}{ha}{B}\opp
\ar@{->}[r]^-{\rhom_{B}(\mathrm{?},K)}
&\dcatdf{hf}{ha}{\mathsf{E}(B)}
}
 \end{gathered}
 \end{equation}

Now from the upper rectangle we deduce that $\Hom_A(F\avlotimes{\tau}{\mathrm?},F)$ is an equivalence, then from the
triangle we see that $\rhom_{A}(\mathrm{?},k)$ is an equivalence, and finally from the lower rectangle we conclude that
$\rhom_{B}(\mathrm{?},K)$ is an equivalence, as desired.

In order to check that $\rhom_{B}(\mathrm{?},K)$ restricts to the equivalences in \eqref{eq:bgg2} it suffices to track
through \eqref{eq:bggC2}, \eqref{eq:morph2}, \eqref{eq:CompositionM}, and \eqref{eq:exactFun2} the equivalences 
obtained by restricting the other functors on the perimeter of the diagram.
  \end{proof}

  \begin{remark}
    \label{rem:right}
With the exception of $\rhom_{B}(\mathrm{?},K)$, each functor on the perimeter of diagram 
\eqref{eq:diagram} has been obtained as part of an adjoint equivalence.  Thus, a right adjoint to
$\rhom_{B}(\mathrm{?},K)$ can be produced as a composition of known functors.  
It can be thought of as the result of applying the theorem to the DG algebra $\mathsf{E}(B)$
and using the canonical quasi-isomorphism $\mathsf{E}(\mathsf{E}(B))\xra{\simeq}B$;
cf.~Corollary \ref{cor:bar-cobar}.
  \end{remark}

 \begin{Notes}
Unlike its covariant cousin $\rhom_B(K,\mathsf{?})\col\dcatdf{}{}{B}\to\dcatdf{}{}{\mathsf{E}(B)^{\mathsf o}}$,
see e.g.\ \cite[\S10]{Ke} or \cite[\S7]{ABIM}, the functor $\rhom_B(?,K)\col\dcatdf{}{}{B}\opp\to\dcatdf{}{}{\mathsf{E}(B)}$ 
does not appear to have been much studied.  We compare Theorem \ref{thm:bgg} to a couple of recent results.

When $\HH(B)$ is augmented and $\ov{\HH(B)}_{\ges0}=0$ Keller and Nicolas \cite[6.4]{KN} prove 
that it restricts to a fully faithful functor $\dcatdf{}{perf}{B}\opp\to\dcatdf{hf}{ha}{\mathsf{E}(B)}$.  If, in addition, 
$\HH_{-1}(B)=0$, then Theorem \ref{thm:bgg} provides much more complete information.  

In \cite[4.7]{HW1} He and Wu prove $\dcatdf{}{perf}{B}\opp\equiv\dcatdf{hf}{ha}{\mathsf{E}(B)}$ for augmented 
DG algebras with $\ov B_{\ges0}=0$, over $K=B/B_{\les-1}$ is isomorphic to a perfect DG module
generated in degree $0$; the only overlap with Theorem \ref{thm:bgg} is when $\HH(B)\cong k$.
  \end{Notes}

\section{Koszul algebras}
  \label{KoszulAlgebras}

We briefly review Koszul algebras from the point of view of twisting cochains.  

  \begin{notation}
    \label{ch:notation9}
Here $A$ is a degreewise finite augmented graded algebra satisfying
  \[
\mathrm{(p)}\quad \ov A_{\les0}=0
\qquad\text{respectively}\qquad
\mathrm{(n)}\quad \ov A_{\ges-1}=0
  \]
Let $\varepsilon^A\col A\to k$ be the canonical augmentation and define vector spaces as follows:
  \[
\text{$V\subseteq\ov A$, which maps isomorphically to }\ov A/{\ov A}{\hskip1pt}^2\,,
  \quad\text{and }\quad
W=\shift V
  \]
  \end{notation}

   \begin{bfchunk}{Two-homogeneity.}
    \label{ch:two-hom}
The inclusion $V\subseteq A$ defines a surjective homomorphism of graded 
algebras $\pi\col \ten aV\to A$.  It define the vertical maps in the following 
commutative diagram, where the horizontal ones are induced by the product of~$A$:
   \[
\xymatrixcolsep{2pc}
\xymatrixrowsep{2pc}
\xymatrix{
V^{\otimes 2}
\ar@{->>}[r]
\ar@{->}[d]_-{\cong}
&V^2
\ar@{->>}[d]
  \\
(\ov A/{\ov A}{\hskip1pt}^2)^{\otimes 2} 
\ar@{->>}[r]
& {\ov A}{\hskip1pt}^2/{\ov A}{\hskip1pt}^3
}
    \]

The algebra $A$ is said to be \emph{two-homogeneous} if the right-hand vertical map is bijective; 
that is, if ${\ov A}{\hskip1pt}^2=V^2\oplus{\ov A}{\hskip1pt}^3$.  The algebra $A$ is \emph{quadratic} if the two-sided ideal
$\Ker(\ten aV\to A)$ is generated by elements of $V^{\otimes 2}$.  One easily sees
that neither property depends on the choice of~$V$, and that quadratic algebras are two-homogeneous.  
   \end{bfchunk}
   
  \begin{bfchunk}{Quadratic dual.}
    \label{ch:quadratic}
Define a $k$-linear map 
  \begin{equation}
    \label{eq:quadratic}
\phi\col W^{\otimes 2}\to\shift^2 ({\ov A}{\hskip1pt}^2/{\ov A}{\hskip1pt}^3)
  \quad\text{by}\quad 
\susp{v}\otimes\susp{v'}\mapsto(-1)^{|v|}\susp^2(vv')
 \end{equation}
As $W$ is degreewise finite, for each $p$ we canonically identify $(W^{\otimes p})^*$ and $(W^*)^{\otimes p}$,
see \eqref{eq:cx2}.  Thus, we obtain an injective morphism of graded vector spaces
  \[
\phi^*\col (\shift^2 ({\ov A}{\hskip1pt}^2/{\ov A}{\hskip1pt}^3))^*\to(W^*)^{\otimes 2}
  \]

The \emph{quadratic dual} of $A$ is the graded algebra\footnote{This is the \emph{opposite} algebra of the 
quadratic dual in \cite[2.8.1]{BGSo}, which is constructed by using the canonical map 
$V^*\otimes U^*\to (U\otimes V)^*$, see \cite[2.7]{BGSo}, rather than the map $\varpi^{UV}$ from \ref{ch:cxDual}(3).}
  \[
A^!:=\ten a{W^*}/\ten a{W^*}\im(\phi^*)\ten a{W^*}
  \]
 \end{bfchunk}

  \begin{bfchunk}{Priddy construction.}
    \label{ch:barPri}
Borrowing notation from~\cite{LV} we write $A\oshriek$ for the graded coalgebra $(A^!)^*$.  It is a subcoalgebra of 
$(\ten a{W^*})^*=\ten c{W}$, which is itself a subcoalgebra of the bar construction $\babar A$.  Recall from \ref{ch:Bar}(2) that 
   \begin{equation}
     \label{eq:barPri1}
\dd^{\babar A}[a_1|\cdots| a_p]=
\sum_{i=1}^{p-1}(-1)^{|a_1|+\cdots+|a_i|+i-1}[{a}_1|\cdots|{a}_{i}a_{i+1}|\cdots| a_p]
  \end{equation}
  
Assume that $A$ is two-homogeneous.  Then $\dd^{\babar A}$ induces to a $k$-linear map
 \[
\ten c{W}\xra{\dd}\bigoplus_{2\les i\les p\les\infty}W^{\otimes(i-2)}\otimes \shift^2(V^2)\otimes W^{\otimes(p-i)}
  \]
Formulas \eqref{eq:barPri1} and \eqref{eq:quadratic} imply $\im(\dd^*)=
\ten a{W^*}\im(\phi^*)\ten a{W^*}$, and hence $\Coker(\dd^*)=A^!$.  We use this fact in two ways.  

On the one hand, since $\dd^{(\babar A)^*}(W^*)=0$ we obtain inclusions of graded algebras
  \begin{equation}
     \label{eq:barPri2}
A^!\cong k\langle W^*\rangle\subseteq\HH((\babar A)^*)=\Ext_A(k,k)
  \end{equation}
where $k\langle W^*\rangle$ denotes the $k$-subalgebra of generated by $W^*$.  

On the other hand, we get $A\oshriek=\Ker(\dd)$, so $A\oshriek$ with zero differential is a DG 
subcoalgebra of $\babar A$. Thus, restricting $\tau^A$ to $A\oshriek$ produces a twisting map
   \begin{equation}
     \label{eq:barPri3}
\tau^\mathsf{p}\col A\oshriek\to A
  \quad\text{given by}\quad
\tau^\mathsf{p}(w)=
  \begin{cases} 
  \susp^{-1}(w) &\text{for } w\in W \\
  0 &\text{otherwise} 
  \end{cases}  
     \end{equation}

We call the left DG $A$-module $A\rotimes{\tau^\mathsf{p}}A\oshriek$ the
\emph{Priddy construction} of $A$.  
  \end{bfchunk}

When $A$ satisfies the conditions in the next theorem it is said to be \emph{Koszul}.

  \begin{theorem}
    \label{thm:koszul}
Let $A$ be a degreewise finite, augmented graded algebra such that
 \[
\mathrm{(p)}\quad \ov A_{\les0}=0
\qquad\text{respectively}\qquad
\mathrm{(n)}\quad \ov A_{\ges-1}=0
  \]

When $A$ is two-homogeneous the following conditions are equivalent.  
  \begin{enumerate}[\quad\rm(i)]
   \item 
The DG algebra $\rhom_A(k,k)$ is formal and $\Ext_A(k,k)\cong A^!$.
   \item 
The graded algebra $\Ext_A(k,k)$ is generated by $W^*$.
    \item 
There exists an acyclic twisting map $A\oshriek\to A$.
   \item 
The twisting map $\tau^\mathsf{p}\col A\oshriek\to A$ from \eqref{eq:barPri3} is acyclic.
  \end{enumerate}
  \end{theorem}

  \begin{proof} 
(i)$\implies$(ii).
This follows immediately from \eqref{eq:barPri2}.

(ii)$\implies$(iv).
From \eqref{eq:barPri2} we get the second equality in the string
  \[
\rank_kA\oshriek_i=\rank_kA_{-i}=\rank_k\HH_{-i}((\babar A)^*)=\rank_k\HH_i(\babar A)
  \]
Thus, $A\oshriek\hookrightarrow\babar A$ is a quasi-isomorphism,
so $\tau^{\mathsf{p}}$ is acyclic by Theorem~\ref{thm:quismAcy}.

(iii)$\implies$(i).  
Lemma \ref{lem:dualCompositionA} gives a quasi-isomorphism $A^!\simeq\End_A(A\rotimes{\tau}C)$.
Since $A\rotimes{\tau}C\simeq k$ is a semifree resolution by \ref{ch:naturalRes}, Lemma 
\ref{ch:CompositionM} gives the desired assertion.
 \end{proof}

The Priddy construction \ref{ch:barPri} acquires a particularly simple form when $A$ is finitely generated as
a $k$-algebra.

  \begin{bfchunk}{Koszul construction.}
    \label{ch:barKos}
Assume that $A$ is two-homogeneous and $V$ has a basis $\{v_1,\dots,v_q\}$.  Let $\{\xi_1,\dots,\xi_q\}$ 
be the of $W^*$ dual to $\{\susp{v_1},\dots,\susp{v_q}\}$ and set
  \begin{equation}
    \label{eq:barKos}
d=\sum_{i=1}^q v_i\otimes \xi_i\in V\otimes W^*\subseteq (A\otimes A^!)_{-1}
  \end{equation}

Let $\sigma\col A\otimes A^!\to\Xi^{\mathsf{o}}_{A^{\scriptstyle{\text{\rm<}}}A}$ denote the map 
$\sigma^{A\avop A^!\avop}$ from Proposition \ref{prop:morphisms}.  Comparison of \eqref{eq:barKos}
and \eqref{eq:barPri2} yields $\sigma(d)=\tau^{\mathsf{p}}$.  The latter is a twisting map, so we get
  \[
\sigma(d^2)=\tau^{\mathsf{p}}\scup\tau^{\mathsf{p}}=\dd^A\tau^{\mathsf{p}}+\tau^{\mathsf{p}}\dd^{A^{\scriptstyle{\text{\rm<}}}}=0
  \]
as both $A$ and $A^{\text{\rm<}}$ have zero differentials.  Since $\sigma$ is an injective morphism of
DG algebras, we conclude that $d^2=0$ holds.

The \emph{Koszul construction} of $A$ is the left DG $A$-module $\koszul A$ defined by
   \[
(\koszul A)\nat=A\otimes A\oshriek
  \quad\text{and}\quad
\dd^{\koszul A}(a\otimes\zeta)=(a\otimes \zeta)d 
  \]
It can be obtained by totaling the classical \emph{Koszul complex} of \cite{Pr}, \cite{Lo}.
  \end{bfchunk}

  \begin{Notes}
The concept of Koszul algebra has been so successful that it has become difficult
to compare the different flavors in use.  Priddy \cite{Pr} originally defined them as 
quadratic algebras satisfying condition (ii) in Theorem \ref{thm:koszul}. L\"ofwall 
\cite{Lo} relaxed the first condition to two-homogeneity.  In \cite{Pr} and \cite{Lo} 
one finds all the results in this section that do not refer to twisting maps; their use 
streamlines the arguments.
  \end{Notes}

\section{Golod DG algebras}
  \label{GolodDGAlgebras}

In this section $B$ stands for a DG algebra.  

We  determine the augmented DG algebras with free Ext algebras.
  
  \begin{bfchunk}{Trivial Massey operations.}
    \label{ch:massey}
A subset $\bsh$ of $\ov{\HH(B)}$ is said to admit a \emph{trivial Massey operation}
if there exists a map $\omicron\col\bigsqcup_{p=1}^\infty\bsh^p\to B$ satisfying the conditions
  \begin{align}
    \label{eq:massey1}
\dd^B(o(h))&=0\text{ and }\cls(\omicron(h))=h\quad \text{for all }h\in\bsh 
  \\
    \label{eq:massey2}
\dd^{B}\omicron(h_{1},\dots,h_{p})
&=\sum_{j=1}^{p-1}(-1)^{|h_{1}|+\dots+|h_{j}|+j}{\omicron}(h_{1},\dots,h_{j})\omicron(h_{{j+1}},\dots,h_{p})
  \\  \notag
&{\phantom{=0}}\text{for all }p\ge2 \text{ and all }(h_{1},\dots,h_{p})\in\bsh^p 
\intertext{Note that the preceding equalities imply that each $(h_{1},\dots,h_{p})\in\bsh^p$ satisfies}
    \label{eq:massey3}
|\omicron(h_{1},\dots,h_{p})|&=|h_{1}|+\dots+|h_{p}|+p-1\quad\text{ for every }p\ge1 
  \end{align}
  \end{bfchunk}

Trivial Massey operations are related to twisting maps as follows.

  \begin{lemma}
    \label{lem:massey}
Let $\bsh=\{h_i\}_{i\in I}$ be a subset of $\ov{\HH(B)}$, let $\omicron\col \bigsqcup_{p=1}^\infty\bsh^p\to B$ 
be a map satisfying \eqref{eq:massey3} and $W$ a graded vector space with basis $\{w_h : |w_h|=|h|+1\}_{h\in\bsh}$.

The $k$-linear homomorphism $\tau^{\omicron}\col\ten c{W}\to B$ defined by $\tau^{\omicron}(1)=0$ and
  \begin{equation}
    \label{eq:golod1}
\tau^{\omicron}({w_{1}}\otimes\dots\otimes {w_{p}})=\omicron(h_{1},\dots,h_{p})\quad\text{for}\quad p\ge1
  \end{equation}
is a twisting map if and only if $\omicron$ is a trivial Massey operation on $\bsh$.
  \end{lemma}

  \begin{proof} 
From $|{w_{1}}\otimes\dots\otimes {w_{p}}|=|h_{1}|+\dots+|h_{p}|+p$ and \eqref{eq:massey3} we get $|\tau^{\omicron}|=-1$.

We first assume that $\omicron$ is a trivial Massey operation and show that $\tau^{\omicron}$ is twisting.  As 
$\ten cW$ has zero differential, we have to check that 
$\dd^{B}\tau^{\omicron}$ and $\varphi^{B}(\tau^{\omicron}\otimes\tau^{\omicron})\psi^{\ten c{W}}$ 
agree on a basis of $\ten c{W}$.  Both return $0$ when evaluated at $1$, as $\psi^{\ten c{W}}(1)
=1\otimes1$ and $\tau^{\omicron}(1)=0$.  For each $h\in\bsh$ we obtain $\dd^{B}\tau^{\omicron}(h)=0$ from \eqref{eq:massey1}, 
and also
  \[
\varphi^{B}(\tau^{\omicron}\otimes\tau^{\omicron})\psi^{\ten c{W}}(h)=\varphi^{B}(\tau^{\omicron}\otimes\tau^{\omicron})(h\otimes1+1\otimes h)=0
  \]
by the expression for $\psi^{\ten c{W}}$ in \ref{ch:tenc}.  If $p\ge2$, then from there and
\eqref{eq:massey2} we get
  \begin{align*}
\dd^{B}\tau^{\omicron}({w_{1}}\otimes\dots\otimes {w_{p}})
&=\sum_{j=1}^{p-1}(-1)^{|h_{1}|+\dots+|h_{j}|+j}\omicron(h_{1},\dots,h_{j})\omicron(h_{{j+1}},\dots,h_{p})
  \\
&=\varphi^{B}(\tau^{\omicron}\otimes\tau^{\omicron})\psi^{\ten c{W}}({w_{1}}\otimes\dots\otimes {w_{p}}) 
  \end{align*}

The converse assertion is verified by reversing the preceding computations.
  \end{proof}

  \begin{bfchunk}{Trivial extensions.}
    \label{ch:trivial}
The \emph{trivial extension} of $k$ by a graded vector space $V$ is the augmented graded algebra $k\ltimes V$ with 
underlying vector space $k\oplus V$, product $(a,v)(a',v')=(aa',av'+a'v)$, unit $(1,0)$, and augmentation $(a,v)\mapsto a$.  

Formula \eqref{eq:barPri1} yields $(\babar{(k\ltimes V)})\nat=\ten c{\shift V}$ and $\dd^{\babar{(k\ltimes V)}}=0$.
  \end{bfchunk}

We say that $B$ is \emph{Golod} if it satisfies the conditions of the next theorem.

    \begin{theorem}
    \label{thm:golod}
Let $B$ be a degreewise finite, augmented DG algebra such that
  \[
\text{\rm(hp)}\quad
\ov{\HH(B)}_{\les0}=0 
\qquad\text{respectively}\qquad
\text{\rm(hn)}\quad 
\ov{\HH(B)}_{\ges-1}=0
  \]
and let $K$ be a left DG $B$-module with $\rank_k\HH(K)=1$; see Theorem \emph{\ref{thm:bgg}}.

The following conditions on $V=\ov{\HH(B)}$ and $W=\shift V$ are equivalent.
  \begin{enumerate}[\quad\rm(i)]
     \item
The graded algebra $\Ext_B(K,K)$ is isomorphic to $\ten a{W^*}$.
   \item
The graded algebra $\Ext_B(K,K)$ is free associative.
    \item 
There exists an acyclic twisting map $\tau\col\ten c{W}\to B$.
    \item 
The DG algebras $B$ is formal and $\HH(B)\cong k\ltimes V$.
    \item 
Some $k$-basis of\ \ $V$ admits a trivial Massey operation.
    \item 
Every $k$-basis of\ \ $V$ admits a trivial Massey operation.
  \end{enumerate}
   \end{theorem}

The hypotheses of the theorem are kept for the rest of the section.  

  \begin{lemma}
    \label{lem:golod}
Assume that $B$ is linked to $B'$ by a string of quasi-isomorphisms of DG algebras, let 
$\bsh$ be a subset of $\ov{\HH(B)}$, and $\bsh'$ the corresponding subset of $\ov{\HH(B')}$.

The set $\bsh$ admits a trivial Massey operation if and only if $\bsh'$ does.
  \end{lemma}

  \begin{proof}
By applying \cite[4.5]{FHT1} we can find a string of quasi-isomorphisms
  \[
B=B^{(0)}\xla{\simeq} B^{(1)}\xra{\simeq}\cdots\xla{\simeq} B^{(2r-1)}\xra{\simeq} B^{(2r)}=B'
  \]
of DG algebras where all arrows pointing left are surjective and under which $\bsh'$ still 
corresponds to $\bsh$.  Thus, we may assume that $\beta\col B\to B'$ is a quasi-isomorphism 
of DG algebras and $\bsh'=\HH(\beta)(\bsh)$, and either $\omicron$ exists or $\omicron'$ 
exists and $\beta$ is surjective.

If $\omicron$ is a trivial Massey operation on $\bsh$, then $\omicron'(h'_{1},\dots,h'_{p})=
\beta\omicron(h_{1},\dots,h_{p})$ defines a trivial Massey operation $\omicron'$ on $\bsh'$.

If $\omicron'$ is a trivial Massey operation on $\bsh'$ and $\beta$ is surjective, then for each $h\in\bsh$ 
we may choose $\omicron(h)\in\ZZ(B)$ with $\beta\omicron(h)=\omicron'(\HH(\beta)(h))$ because
$\ZZ(\beta)\col \ZZ(B)\to\ZZ(B')$ is surjective.  This yields a map $\omicron\col\bsh\to B$ satisfying 
\eqref{eq:massey1} and $\beta\omicron=\omicron'$.  

Assume, by induction, that for some $p\ge2$ an 
extension $\omicron\col\bigsqcup_{j=1}^{p-1}\bsh^j\to B$ has been produced, so that \eqref{eq:massey2} 
holds and $\beta\omicron=\omicron'$.  For each $(h_{1},\dots,h_{j})\in\bsh^p$ set
  \[
z(h_{1},\dots,h_{j})
=\sum_{j=1}^{p-1}(-1)^{|h_{1}|+\dots+|h_{j}|+j}\omicron(h_{1},\dots,h_{j})\omicron(h_{1},\dots,h_{j})\,.
  \]
We then have $\dd^{B}(z(h_{1},\dots,h_{j}))=0$ and $\beta(z(h_{1},\dots,h_{j}))=\dd^B(\omicron'(h'_{1},\dots,h'_{p}))$.

Choose $y(h_{1},\dots,h_{j})$ in $B$ so that $\beta(y(h_{1},\dots,h_{j}))=\omicron'(h'_{1},\dots,h'_{j})$.  
Now the element $z(h_{1},\dots,h_{j})-\dd^{B}(y(h_{1},\dots,h_{j}))$ of $B$ is a cycle in $\Ker(\beta)$.  
As $\beta$ 
is a surjective quasi-isomorphism, we have $\HH(\Ker(\beta))=0$, so $z(h_{1},\dots,h_{j})-\dd^{B}(y(h_{1},\dots,h_{j}))$ 
is the boundary of some $x(h_{1},\dots,h_{p})$ in $\Ker(\beta)$.  The element
  \[
\omicron(h_{1},\dots,h_{p})\coloneqq x(h_{1},\dots,h_{p})+y(h_{1},\dots,h_{p})
  \]
satisfies $\dd^{B}(\omicron(h_{1},\dots,h_{p}))=z(h_{1},\dots,h_{p})$ and 
$\beta(\omicron(h_{1},\dots,h_{p}))=\omicron'(h'_{1},\dots,h'_{p})$.  
The inductive construction of a trivial Massey operation on $\bsh$ is now complete.
  \end{proof}

  \begin{proof}[Proof of Theorem \emph{\ref{thm:golod}}]
(ii)$\implies$(iv).  
Using \ref{ch:models}, choose a quasi-isomorphism of DG algebras $\alpha\col A\to B$
so that $A$ is degreewise finite, augmented, with $\ov{A}_{\les0}=0$, respectively, $\ov{A}_{\ges-1}=0$. 
The hypothesis means $\Ext_B(K,K)=\ten a{U}$ for some graded vector space,~$U$. 
By Lemmas \ref{lem:simple} and \ref{lem:exactFun}, this carries over when we replace 
$B$ with $A$ and $K$ with $k$.  The desired property evidently does not change under these substitutions, so
we may assume that $B$ is augmented, degreewise finite and has $\ov B_{\les0}=0$, respectively,
$\ov B_{\ges-1}=0$.  From $\Ext_B(K,K)\cong\HH((\babar B)^*)$ we see that~$U$ is degreewise 
finite with $U_{\ges-1}=0$, respectively, $U_{\les0}=0$.  We have a surjection
  \[
\ZZ((\babar B)^*)\twoheadrightarrow\HH((\babar B)^*)=\Ext_B(K,K)=\ten a{U}
  \]
 of graded algebras.  Picking a right inverse and composing it with $\ZZ((\babar B)^*)\subseteq(\babar B)^*$ we 
get a quasi-isomorphism of augmented DG algebras $\ten a{U}\xra{\simeq}(\babar B)^*$.
This yields a quasi-isomorphism $\ten c{U^*}\xla{\simeq}\babar{B}$ of DG coalgebras.  Since \ref{ch:trivial}
gives $\ten c{U^*}=\babar{(k\ltimes\shift^{-1}(U^*))}$ we get a string of quasi-isomorphism of DG algebras
  \[
B\xla{\simeq}\cobar{\babar{B}}\xra{\simeq}\cobar(\ten c{U^*})=\cobar{\babar{(k\ltimes\shift^{-1}(U^*)})}\xra{\simeq}k\ltimes\shift^{-1}(U^*)
  \]
where the first and third come from Corollary \ref{cor:bar-cobar} and the second from \ref{ch:BarCobarNat}.
The DG algebra on the right has zero differential, so we obtain $\shift^{-1}(U^*)\cong\ov{\HH(B)}=V$.

(iv)$\implies$(vi).
Let $\bsh$ be any basis of $\ov{\HH(B)}$, and $\bsh'$ the basis of the subspace $V$ of $k\ltimes V$ corresponding to
it through the given string of quasi-isomorphisms. Setting $\omicron(h')=h'$ and $\omicron(h'_1,\dots,h'_p)=0$ 
for $p\ge2$ we get a trivial Massey operation on~$\bsh'$.  By Lemma \ref{lem:golod}, there is a corresponding trivial 
Massey operation on $\bsh$.

(v)$\implies$(iii).
Let $\omicron$ be a trivial Massey operation on a basis $\bsh$ of $V$.  

We choose for $W=\shift V$ the basis $\{w_h=\susp(h)\}_{h\in\bsh}$ and proceed to show that
the twisting map $\tau=\tau^{\omicron}$ from \eqref{eq:golod1} is acyclic.  Indeed, 
\eqref{eq:twistXi2} and \eqref{eq:golod1} give
  \begin{equation}
   \begin{aligned}
    \label{eq:golod2}
\dd^{B\rotimes{\tau}\ten c{W}}(b\otimes {w_{1}}\otimes\dots\otimes {w_{p}})
=\dd^B(b)\otimes &\, {w_{1}}\otimes\dots\otimes {w_{p}}
    \\
+(-1)^{|b|}\sum_{j=1}^{p}b\,\omicron(h_{1},\dots,h_{j})\otimes &\, {w_{j+1}}\otimes\dots\otimes {w_{p}}
   \end{aligned}
     \end{equation}
Thus, $b\mapsto b\otimes1$ defines an exact sequence of complexes
  \[
0\to B\to B\rotimes{\tau}\ten c{W}\to(B\rotimes{\tau}\ten c{W})\otimes {W}\to0
  \]
From \eqref{eq:massey1} and \eqref{eq:golod1} we see that the connecting map in its 
homology sequence sends $\cls(1\otimes1\otimes{w_h})$ to $h$ for each $h$ in $\bsh$.  Thus,
setting $H_n=\HH_n(B\rotimes{\tau}\ten c{W})$ we get an isomorphism $H_0\cong k$, 
and for each integer $n$ we obtain an exact sequence 
  \[
0\to H_n\to \bigoplus_{i\in\BZ}(H_{n-i}\otimes W_i)\xra{\pi} \HH_{n-1}(\ov B)\to0
  \]
of vector spaces, where $\pi|_{H_0\otimes W_n}$ is an isomorphism.  If (p) holds, then
$H_n=0$ for $n\le-1$ and ${W}_{i}=0$ for $i\le1$, so the middle term equals 
$\bigoplus_{i=1}^{n}H_{n-i}\otimes W_i$.  For $n=1$ it is zero, hence $H_1=0$. 
When $n\ge2$ we may assume $H_i=0$ holds for $1\le i<n$.  The middle term 
then equals $H_0\otimes W_n$, which forces $H_n=0$.  If (n) holds, then 
$H_n=0$ for $n\ge1$ and ${W}_{i}=0$ for $i\ge0$, and a similar argument applies.

(iii)$\implies$(i).  
By \ref{eq:dualCompositionA3} we have $\Ext_B(K,K)=\HH(\ten a{W^*})=\ten a{W^*}$.
  \end{proof}

  \begin{Notes}
Golod~\cite{Go} proved that when $B$ is the Koszul complex of a commutative local ring $R$ the 
complex in \eqref{eq:golod2} is a minimal resolution of the residue field of~$R$.  A version of the 
theorem for graded-commutative DG algebras $B$ is proved in~\cite{Av}:  In (i) and (ii) the 
isomorphisms involving $\Ext_A(k,k)$ are maps of \emph{Hopf algebras} (the  tensor algebras 
being primitively generated), and in (iv) formality is a achieved through strings of 
quasi-isomorphisms of \emph{graded-commutative} DG algebras.  
  \end{Notes}

\appendix

\section{Complexes}

In the appendices, as in the rest of the paper, we work over a fixed field $k$.

The material on complexes is standard.  Some choices need to be made at an early
stage, especially when signs are involved; they are described explicitly with the intent
to be applied consistently throughout the text.  Some attention 
is given to recording conditions for bijectivity of a number of canonical morphisms; 
everything is completely elementary, but small variations occur and occasionally do matter.

  \begin{bfchunk}{Complexes.}
    \label{ch:cxCx}
A complex $V$ is a sequence of $k$-linear maps $\dd^V_i\col V_i\to V_{i-1}$ indexed by 
$i\in\BZ$ and satisfying $\dd^V_i\dd^V_{i+1}=0$ for each $i$.  We write $\dd^V$ for the sequence 
$(\dd^V_i)_{i\in\BZ}$ and refer to it as the differential of $V$.  The elements of $V_i$ are called elements 
of $V$ of degree $i$; our preferred notation is $v\in V$ with $|v|=i$.  

The \emph{shift} of $V$ is the complex $\shift V$ with $(\shift V)_i=V_{i-1}$ and 
$\dd^{\shift V}_i=-\dd^V_{i-1}$ for all~$i$.  For each element $v\in V_i$ 
we let $\susp v$ denote the element $v\in(\shift V)_{i+1}$.

We say that a complex $V$ is \emph{degreewise finite} if $\rank_kV_i$ is finite 
for each $i$, and that it is \emph{finite} if, in addition $V_i=0$ for $|i|\gg0$.

A graded vector space is a complex with zero differential.  We let $V\nat$ denote the graded vector
space underlying a complex $V$.
  \end{bfchunk}
  
For the rest of this appendix $U$, $V$, and $W$ denote complexes of vector spaces.  

  \begin{bfchunk}{Homomorphisms.}
    \label{ch:cxHom}
A homomorphism $\upsilon\col U\to V$ of degree $p$ is a family $\upsilon=(\upsilon_i)_{i\in\BZ}$
of $k$-linear maps $\upsilon_i\col U_i\to V_{i+p}$; that is, an element of
the space $\Hom(U,V)_p=\prod_{j-i=p}\Hom(U_i,V_j)$.  Composed in the obvious
way with a homomorphism $\omega\col V\to W$ of degree $q$, it yields a 
homomorphism $\omega\upsilon\col U\to W$ of degree $p+q$.  Thus, $\dd^V$ is a 
homomorphism $V\to V$ of degree $-1$, satisfying $(\dd^V)^2=0$.

The complex $\Hom(U,V)$ has $\Hom(U,V)_p$ as $p$th component and differential
  \begin{equation}
   \label{eq:cxHom}
\dd^{\Hom(U,V)}(\upsilon)=\dd^V\upsilon-(-1)^{|\upsilon|}\upsilon\dd^U 
  \end{equation}
The sign $(-1)^{|\upsilon|}$ is mandated by the rule that a coefficient  $(-1)^{|x||y|}$ 
appear in formulas whenever symbols $x$ and $y$ switch places: here $|\dd^V|=-1$.

A \emph{chain map} is a homomorphism $\upsilon\col U\to V$ that is a cycle in $\Hom(U,V)$; that is,
$\dd^V\upsilon=(-1)^{|\upsilon|}\upsilon\dd^U$ holds; $v\mapsto\susp v$ is a chain map 
$\susp \col V\to\shift V$ with $|\susp|=1$.

A \emph{morphism} of complexes is a chain map of degree~$0$. 

A \emph{quasi-isomorphism} is a morphism $\upsilon$, such that $\HH(\upsilon)$ is an isomorphism.
This property is indicated by the symbol $\simeq$ while $\cong$  is reserved for isomorphisms.
   \end{bfchunk}

  \begin{bfchunk}{Tensor products.}  
      \label{ch:cxTensor}
The complex $U\otimes V$ has $(U\otimes V)_n=\bigoplus_{i+j=n}U_i\otimes V_j$ and 
  \begin{equation}
    \label{eq:cxTen}
\dd^{U\otimes V}=\dd^U\otimes V+U\otimes\dd^V 
  \end{equation}
Evaluating both sides on $u\otimes v$ yields $\dd^{U\otimes V}(u\otimes v)=
\dd^U(u)\otimes v+(-1)^{|u|}u\otimes\dd^V(v)$.

When $\upsilon\col U\to V$ and $\upsilon'\col U'\to V'$ are homomorphisms of 
degree $p$ and $p'$, respectively, a natural homomorphism of complexes
  \begin{align}
    \label{eq:cx0}
\upsilon\otimes\upsilon'\col U\otimes U'\to V\otimes V'
    \end{align}
of degree $p+p'$ is defined by setting $(\upsilon\otimes\upsilon')(u\otimes u')
=(-1)^{|\upsilon'||u|}\upsilon(u)\otimes\upsilon'(u')$.  The definition implies the  
follows equalities:
   \begin{align}
    \label{eq:cx0.5}
\upsilon\otimes\upsilon'
=(\upsilon\otimes U')(U\otimes\upsilon')
=(-1)^{|\upsilon||\upsilon'|}(U\otimes\upsilon')(\upsilon\otimes U') 
    \end{align}
      \end{bfchunk}

  \begin{bfchunk}{Dual complexes.}
    \label{ch:cxDual}
By abuse of notation, $k$ stands also for the complex that has a unique non-zero component, which appears 
in degree $0$ and is equal to $k$.  

(1)  Set $U^*=\Hom(U,k)$; thus, $(U^*)_i=\Hom(U_{-i},k)$ and $\dd^{U^*}(\alpha)=(-1)^{|\alpha|-1}\alpha\dd^U$ 
for $\alpha\in U^*$. For each homomorphism $\upsilon\col U\to V$, a natural homomorphism 
$\upsilon^*\col V^*\to U^*$ with $|\upsilon^*|=|\upsilon|$ is defined by
$\upsilon^*(\beta)=(-1)^{|\upsilon||\beta|}\beta\upsilon$ for $\beta\in V^*$.

For every homomorphism $\omega\col V\to W$ the following equality holds:
  \[
(\omega\upsilon)^*=(-1)^{|\upsilon||\omega|}\upsilon^*\omega^* 
  \]

Applied with $\upsilon=\dd^V$ and $\omega=\beta\in\Hom(V,k)$, this formula yields
  \begin{equation}
    \label{eq:cx1.5}
(\dd^{V})^*=-\dd^{V^*} 
  \end{equation}

(2)  The formula $\delta^{UV}(\upsilon)=\upsilon^*$ gives a natural morphism
  \begin{align}
    \label{eq:cx3}
\delta^{UV}\col\Hom(U,V)&\to\Hom(V^*,U^*)
    \end{align}
 which is always injective; it is bijective if $V$ is degreewise finite.  

Indeed, using formulas \eqref{eq:cxHom}, \eqref{eq:cx1.5}, and \eqref{eq:cx3} one easily sees that 
$\delta^{UV}$ commutes with differentials.  Note that $\delta^{UV}_n$ is the natural $k$-linear map
  \begin{align*}
\prod_{j-i=n}\Hom(U_i,V_j)\to\prod_{i-j=n}\Hom((V_{-j})^*,(U_{-i})^*)
  \end{align*}
which is always injective, and is bijective when $\rank_kV_j$ is finite for each $j\in\BZ$.

(3)  Setting $\varpi^{UV}\!(\alpha\otimes\beta)(u\otimes v)=(-1)^{|\beta||u|}\alpha(u)\beta(v)$
one gets a natural morphism
  \begin{align}
    \label{eq:cx2}
\varpi^{UV}\col U^*\otimes V^*&\to(U\otimes V)^*
    \end{align}
that is always injective; it is bijective if $U$ is finite, or if $V$ is finite, or if one of 
$U$ and $V$ is degreewise finite and $U_{\ll0}=0=V_{\ll0}$ or $U_{\gg0}=0=V_{\gg0}$ holds.

Indeed, using \eqref{eq:cxTen}, \eqref{eq:cx1.5}, and \eqref{eq:cx2} one sees that $\varpi^{UV}$ commutes with 
differentials.  Note that $\varpi^{UV}_n$ is a composition of natural injective maps
  \begin{align*}
\bigoplus_{-i-j=n}(U^*\otimes V^*)_n\to\bigoplus_{i+j=-n}(U_i\otimes V_j)^*\subseteq\prod_{i+j=-n}(U_i\otimes V_j)^*
  \end{align*}
Under the additional hypotheses the arrow on the left is bijective, while the product has only
finitely many non-zero terms, so the inclusion is an equality.
     \end{bfchunk}

\section{DG modules}
  \label{app:DGA}

This appendix deals with standard and widely available material, so the main purpose 
is to fix terminology and notation.  As a last item we present a very simple abstract 
algebraic formalism underlying the construction of twisted tensor products.

  \begin{bfchunk}{DG algebras.}
    \label{ch:DGAaug}
A DG algebra $A$ is a complex endowed with morphisms $\varphi^A\col A\otimes A\to A$ 
(the \emph{product}) and $0\ne\eta^A\col k\to A$ (the \emph{unit}), satisfying
  \[
\varphi^A(\varphi^A\otimes A)=\varphi^A(A\otimes\varphi^A)
  \quad\text{and}\quad
\varphi^A(\eta^A\otimes A)=\id^A=\varphi^A(A\otimes\eta^A) 
  \]

As usual, we set $ab=\varphi^{A}(a\otimes b)$ and $1=\eta^A(1)$.
The \emph{opposite DG algebra} $A\avop$ has the same underlying complex and unit 
as $A$, and product $a\cdot b=(-1)^{|a||b|}ba$.  

A \emph{morphism} $\alpha\col A'\to A$ of DG algebras is a morphism of complexes, such that 
$\alpha\eta^{A'}=\eta^{A}$ and $\alpha\varphi^{A'}=\varphi^{A}(\alpha\otimes\alpha)$.  

Each complex $V$ defines a DG algebra $\End(V)$ with underlying complex $\Hom(V,V)$, product given by 
composition of morphisms, and unit $\id^V$.

A \emph{graded algebra} is a DG algebra with zero differential.

In case $A$ satisfies one of the conditions $A_{\ll0}=0$ or $A_{\gg0}=0$, a graded vector space 
$V$ is said to be \emph{adequate} for $A$ if $V_{\ll0}=0$, respectively, $V_{\gg0}=0$ holds.
  \end{bfchunk}

For the balance of this appendix $A$ denotes a DG algebra.

   \begin{bfchunk}{Augmentations.}
      \label{ch:AugA}
An \emph{augmented} DG algebra is a morphism $\varepsilon^A\col A\to k$ of DG algebras; 
with $\ov A=\Ker\varepsilon^A$ one has $A=k1\oplus\ov A$ canonically.  If $A_0=k$ and either 
$A_{\les-1}=0$ and $\dd^A_1=0$, or $A_{\ges1}=0$,  then $A$ has a \emph{unique} augmentation.

Let $\varepsilon^{A'}$ be an augmented DG algebra.  A morphism of \emph{augmented DG algebras} 
is a morphism $\alpha\col A'\to A$ of DG algebras such that $\varepsilon^{A}\alpha=\varepsilon^{A'}$.  

A \emph{derivation} is a $k$-linear map $\delta\col A\to A$ with 
$\delta\varphi^{A}=\varphi^{A}(\delta\otimes A+A\otimes\delta)$.
  \end{bfchunk}

  \begin{bfchunk}{Tensor algebras.}
    \label{ch:tena}
The \emph{augmented tensor algebra} $\ten aV$ of a graded vector space $V$ has underlying space 
$\bigoplus_{p=0}^\infty \tenn pV$ with $\tenn pV=V^{\otimes p}$, augmentation $1\mapsto1$ and 
$v_1\otimes\cdots\otimes v_p\mapsto0$ for $p\ge1$,  unit $1\mapsto1\in V^{\otimes 0}$, and product 
 \[
(v_1\otimes\cdots\otimes v_p)\otimes(v_{p+1}\otimes\cdots\otimes v_{p+q})\mapsto
v_1\otimes\cdots\otimes v_{p+q}
  \]

Let $\varepsilon^{A'}$ be any augmented graded algebra.

For every morphism $\beta\col V\to \ov A'$ of graded $k$-spaces there is a unique morphism 
$\wt\beta\col\ten aV\to A'$ of augmented graded algebras with $\wt\beta|_V=\beta$.

For every $k$-linear map $\delta\col V\to \ten cV$ of degree $d$ there is a unique derivation 
$\wt\delta$ of $\ten aV$ of degree $d$ with with $\wt\delta|_V=\delta$.
  \end{bfchunk}

  \begin{bfchunk}{Left DG modules.}
    \label{ch:DGAmod}
A \emph{left DG $A$-module} is a complex $M$ with a fixed morphism $\varphi^{AM}\col A\otimes M\to M$ 
(the \emph{action} of $A$ on $M$), which satisfies 
  \[
\varphi^{AM}(\eta^A\otimes M)=\id^M
  \quad\text{and}\quad
\varphi^{AM}(A\otimes\varphi^{AM})=\varphi^{AM}(\varphi^A\otimes M) 
  \]
Equivalently, there is a fixed morphism $\lambda^{AM}\col A\to\End(M)$ of DG algebras 
(the \emph{representation} of $A$ in $M$); see \ref{ch:DGAaug}.
The two structures are linked by the formula
  \begin{equation}
    \label{eq:lambda}
\lambda^{AM}(a)(m)=\varphi^{AM}(a\otimes m) 
  \end{equation}

When $V$ is a complex $\varphi^{A(A\otimes V)}:=\varphi^A\otimes V$ turns $A\otimes V$ into a left DG $A$-module.

When $M$ is a left DG $A$-module so is $\shift^sM$, with $a\susp^s(m)=(-1)^{|a|s}\susp^{s}(am)$.

A \emph{homomorphism} $\mu\col M\to M'$ of left DG $A$-modules is a homomorphism of 
complexes satisfying $\mu(am)=(-1)^{|\mu||a|}a\mu(m)$.  The set of all such homomorphisms 
is a subcomplex $\Hom_A(M,M')$ of  $\Hom(M,M')$.  Thus, the notions of chain maps and 
morphisms of complexes, defined as in \ref{ch:cxHom}, restrict to notions for DG modules.

When $\alpha\col A'\to A$ is a morphism of DG algebras each left DG $A'$-module $M'$ 
then has a natural structure of left DG $A$-module.  A map $\mu\col M\to M'$ or 
$\mu'\col M'\to M$ is an $\alpha$-\emph{equivariant morphism} of left DG modules if it is a 
morphism of left $A$-modules.

\emph{Graded modules} are DG modules with zero differentials.
  \end{bfchunk}

  \begin{bfchunk}{Right DG modules.}
    \label{ch:DGArightmod}
A \emph{right} DG $A$-module is a complex $N$ equipped with a morphism $\varphi^{NA}\col N\otimes A\to N$ 
subject to the evident restrictions.It has a canonical structure of left DG module over $A\avop$, see \ref{ch:DGAaug}, 
given by $a\cdot n=(-1)^{|n||a|}na$. In particular, \emph{homomorphisms} $\nu\col N\to N'$ of right DG modules satisfy $\nu(na)=\nu(n)a$. 
  
If $M$ and $N$ are left DG modules over $A$ and $A\avop$, respectively, then 
  \begin{equation}
    \label{eq:coinduced}
\lambda^{AN^*}(\xi)=(\lambda^{A\avop N}(\xi))^*
  \quad\text{and}\quad
\lambda^{A\avop M^*}(\xi)=(\lambda^{AM}(\xi))^*
  \end{equation}
define left representations of $A$ in $N^*$ and of $A\avop$ in $M^*$, respectively.
  \end{bfchunk}

\begin{bfchunk}{Semifree DG modules.}
    \label{ch:DGMhp}
A \emph{semifree filtration} of a DG $A$-module $F$ is a sequence $\cdots\subseteq F^{p-1}\subseteq F^p\subseteq\cdots$ 
of DG submodules, such that $F^p=0$ for $p\ll0$, $\bigcup_p F^p=F$, and for each $p\in\BZ$ there is a complex
$V^p$ with $\dd^{V^p}=0$ and an isomorphism of DG $A$-modules $F^p/F^{p-1}\cong A\otimes V^p$.
A DG $A$-module is \emph{semifree} if it admits some semifree filtration.  See \cite[\S1]{AH} or 
\cite[\S6]{FHT2} for the following properties.

When $F$ is a semifree DG module $F$ both $?\otimes_AF$ and $\Hom_A(F,?)$ preserve quasi-iso\-morphisms.  
Since degree zero cycles in $\Hom$ complexes are morphisms of DG modules, it follows 
that for each quasi-isomorphism $\varkappa^{C}\col L\simeq  K$ and every morphism $\phi\col F\to K$ 
there exists a morphism $\lambda\col F\to L$, such that $\phi$ is homotopic to $\varkappa^{C}\lambda$,
and that such a morphism is unique up to homotopy.  As a consequence, any quasi-isomorphism of semifree 
DG $A$-modules is a homotopy equivalence.
     \end{bfchunk}  

  \begin{bfchunk}{Left DG bimodules.}
    \label{ch:DGAbimod}
When $\Xi$ is a DG algebra, a left $A$-$\Xi$-bimodule is a complex $M'$ 
that is a left DG $A$-module and a left DG $\Xi$-module, and these structures are
compatible: $a(\xi m')=(-1)^{|a||\xi|}\xi(am')$ for $a\in A$, 
$m'\in M'$, and $\xi\in\Xi$. 
Morphisms of DG bimodules commute with all structures in place.

When $N$ is a right DG $A$-module and $M'$ a left DG $A$-$\Xi$-bimodule the formulas 
  \begin{equation}
    \label{eq:biinduced}
(\xi\vartheta)(m)=\xi\vartheta(m)
  \quad\text{and}\quad
\xi(n\otimes m')=(-1)^{|\xi||n|}n\otimes\xi m'
  \end{equation}
define structures of left DG $\Xi$-modules on $\Hom_A(M,M')$ and $N\otimes_AM'$, respectively.
  \end{bfchunk}

\begin{bfchunk}{Twisters.}
    \label{ch:twisters}
Let $\Xi$ be a DG algebra and $\tau$ an element of $\Xi$.

We say that $\tau$ is a \emph{twister} if it satisfies $\dd^{\Xi}(\tau)=\tau^2$; note that then $|\tau|=-1$.
  \begin{enumerate}[\rm(1)]
    \item
Clearly, $\tau$ is a twister if and only if $-\tau$ is a twister in $\Xi\avop$; see \ref{ch:DGAaug}.
    \item
Let $\tau$ be a twister.  For every left DG $\Xi$-module $U$ the map 
  \[
{}^{\tau}\dd=\dd^U-\lambda^{\Xi U}(\tau)\col U\to U
  \]
with $\lambda^{\Xi U}$ from \eqref{eq:lambda} defines a complex ${}^{\tau}U$ with $({}^{\tau}U)\nat=U\nat$,
as one has
\begin{align*}
({}^{\tau}\dd)^2(u)&=(\dd^U)^2(u)-\dd^U({\tau}u)-{\tau}\dd^U(u)+{\tau}^2u=({\tau}^2-\dd^{\Xi}({\tau}))u=0 
  \end{align*}
By (1), a left $\Xi\avop$-module $V$ yields a complex $V^{\tau}$ with $(V^{\tau})\nat=V\nat$ and differential
  \[
\dd^{\tau}=\dd^V+\lambda^{\Xi\avop V}(\tau)\col V\to V 
  \]
    \item
The map $\tau=0$ is a twister, which produces complexes ${}^{0}U=U$ and $V{}^{0}=V$.
    \item
If $M$ is a left DG $A$-$\Xi$-bimodule, then ${}^{\tau}M$ is a left DG $A$-module.  
     \end{enumerate}

Let $\delta\col \Xi\to\Xi'$ be a morphism of DG algebras and set $\tau'=\delta(\tau)$.

  \begin{enumerate}[\rm(1)]
    \item[\rm(5)]
If $\tau$ is a twister, then so is $\tau'$; the converse holds when $\delta$ is injective.
    \item[\rm(6)]
If $\mu\col M\to M'$  is an $A$-$\delta$-equivariant morphism of left DG bimodules, then  
${}^{\tau}\mu\col {}^{\tau}M\to{}^{\tau'}\!M'$ is a morphism of left DG $A$-modules, and 
${\mathrm{?}}\mapsto{}^{\tau}\mathrm{?}$ is a functor.
     \end{enumerate}
       \end{bfchunk}

\section{DG comodules}
  \label{app:DGC}

At a basic level notions concerning coalgebras and their comodules mirror 
the corresponding concepts for algebras and modules.  

  \begin{bfchunk}{DG coalgebras.}
    \label{ch:DGCaug}
A DG coalgebra $C$ is a complex equipped with morphisms $\psi^C\col C\to C\otimes C$ 
(the \emph{coproduct}) and $0\ne\varepsilon^C\col C\to k$ (the \emph{counit}), satisfying
  \[
(\psi^C\otimes C)\psi^C=(C\otimes \psi^C)\psi^C
  \quad\text{and}\quad
(\varepsilon^C\otimes C)\psi^C=\id^C=(C\otimes \varepsilon^C)\psi^C 
  \]

A \emph{morphism} $\gamma\col C\to C'$ of DG algebras is a morphism of complexes satisfying 
$\varepsilon^{C'}\gamma=\varepsilon^C$ and $\psi^{C'}\gamma=(\gamma\otimes\gamma)\psi^{C}$.

A \emph{graded coalgebra} is a DG coalgebra with zero differential.

In case $C$ satisfies one of the conditions $C_{\ll0}=0$ or $C_{\gg0}=0$, a graded vector space 
$V$ is said to be \emph{adequate} for $C$ if $V_{\ll0}=0$, respectively, $V_{\gg0}=0$ holds.
  \end{bfchunk}

For the rest of this appendix $C$ denotes a DG coalgebra.

   \begin{bfchunk}{Coaugmentations.}
      \label{ch:AugC}
A \emph{coaugmented} DG coalgebra is a morphism \text{$\eta^C\col k\to C$} of DG coalgebras.  Set $1=\eta^C(1)$ 
and $\ov C=\Ker(\varepsilon^C)$; thus, $C=k1\oplus\ov C$.  If $C_0=k$ and either $C_{\les-1}=0$ or $C_{\ges1}=0$ 
and $\dd^C_0=0$, then $C$ has a \emph{unique} coaugmentation.

The \emph{reduced coproduct} $\ov{\psi}{}^C\col\ov C\to
\ov C{}^{\otimes2}$ is the map $c\mapsto\psi^{C}(c)-1\otimes c-c\otimes1$.  Set $\ov\psi{}^{C(1)}=\ov\psi{}^{C}$ and 
$\ov\psi{}^{C(p)}:=(\ov{\psi}{}^{C(p-1)}\otimes \ov C)\ov{\psi}{}^C\col\ov C\to\ov C{}^{\otimes p}$ for $p\ge2$.
We say that $C$ is \emph{cocomplete} if it is coaugmented and $\bigcup_{p\ge1}\Ker(\ov\psi{}^{C(p)})=\ov C$.  
This holds, for example, if either $\ov C_{\les0}=0$ or $\ov C_{\ges0}=0$ holds.

Let $\eta^{C'}$ be a coaugmented DG coalgebra.  A morphism $\gamma\col C\to C'$ of DG coalgebras is one
of \emph{coaugmented} coalgebras if $\gamma\eta^{C}=\eta^{C'}$.  It yields a morphism  of complexes 
$\ov\gamma\col\ov C\to\ov C'$ satisfying $\ov\gamma{}^{\otimes p}\ov\psi{}^{C(p)}=\ov\psi{}^{C'(p)}\ov\gamma$
for all $p\ge1$.  

A \emph{coderivation} is a $k$-linear map $\delta\col C\to C$ with $\psi^{C}\delta|_{\ov C}=
(\delta\otimes C+C\otimes\delta)\ov\psi{}^C$ and $\delta(1)=0$.
  \end{bfchunk}

  \begin{bfchunk}{Tensor coalgebras.}
    \label{ch:tenc}
The \emph{tensor coalgebra} $\ten cV$ of a graded vector space $V$ has underlying graded $k$-space 
$\bigoplus_{p=0}^\infty \tenn pV$ with $\tenn pV=V^{\otimes p}$, coaugmentation $1\mapsto1$, counit 
$1\mapsto1$ and $v_1\otimes\cdots\otimes v_p\mapsto0$ for $p\ge1$, and reduced coproduct 
  \[
\ov\psi(v_1\otimes\cdots\otimes v_{p})=\sum_{j=1}^{p-1} (v_1\otimes\cdots\otimes v_j)\otimes(v_{j+1}\otimes\cdots\otimes v_{p})
  \]
The formula implies $\ov\psi{}^{(i)}(v_1\otimes\cdots\otimes v_{p})=0$ for $i>p$, so $\ten cV$ is cocomplete.

Let $\pi\col\ten cV\to V$ be the projection and $\eta^C$ be any \emph{cocomplete} coalgebra.
 
For every morphism $\beta\col \ov C\to V$ of graded $k$-spaces there is a unique morphism $\wt\beta\col C\to\ten cV$ 
of graded coaugmented coalgebras with $\pi\wt\beta|_{\ov C}=\beta|_{\ov C}$.

For every $k$-linear map $\delta\col\ten cV\to V$ of degree $d$ there is a unique coderivation $\wt\delta$ of $\ten cV$ 
of degree $d$ with $\pi{\wt\delta}|_{V}=\delta$.
  \end{bfchunk}

   \begin{bfchunk}{Left DG comodules.}
    \label{ch:DGAcomod}
A \emph{left DG $C$-module} is a complex $X$ with a fixed morphism $\psi^{CX}\col X\to C\otimes X$, the 
\emph{coaction} of $C$ on $X$, which satisfies 
  \[
(C\otimes\psi^{CX})\psi^{CX}=(\psi^{C}\otimes X)\psi^{CX}
  \quad\text{and}\quad
(\varepsilon^C\otimes X)=\id^X 
  \]

When $V$ is a complex $\psi^{C(C\otimes V)}:=\psi^C\otimes V$ turns $C\otimes V$ into a left DG $C$-comodule.

A \emph{homomorphism} $\chi\col X\to X'$ of left DG comodules is a homomorphism of 
complexes satisfying $(C\otimes\chi)\psi^{CX}=\psi^{CX'}\chi$.  The set of all such homomorphisms 
is a subcomplex $\Hom_C(X,X')$ of  $\Hom(X,X')$.  Thus, the notions of chain maps and 
morphisms of complexes, defined as in \ref{ch:cxHom}, restrict to notions for DG comodules.

When $X$ is a left DG $C$-comodule so is $\shift^sX$, with coproduct defined as follows:
if $\psi^{X}(x)=\sum_ic_i\otimes x_i$, then $\psi^{\shift^sX}(\susp^s(x))=\sum_i(-1)^{|c_i|s}c_i\otimes\susp^{s}(x_i)$.

\emph{Graded comodules} are DG comodules with zero differential.
  \end{bfchunk}

  \begin{bfchunk}{Right DG comodules.}
    \label{ch:DGCrightmod}
A \emph{right} DG $C$-comodule is a complex $Y$ with 
a morphism $\psi^{YC}\col Y\to Y\otimes C$, subject to the evident conditions.
Homomorphisms and related notions are defined by evident alterations of the 
definitions in \ref{ch:DGAcomod}.
  \end{bfchunk}

  \end{document}